\documentclass[a4paper, 11pt, reqno]{amsart}
\usepackage{graphicx,verbatim}

\usepackage{enumitem}
\usepackage{pstricks,tikz}
\usetikzlibrary{positioning,chains,fit,shapes,calc, fit, decorations, decorations.markings}

\newtheorem{theorem}{Theorem}[section]
\newtheorem{lemma}[theorem]{Lemma}
\newtheorem{proposition}[theorem]{Proposition}

\newtheorem{conjecture}[theorem]{Conjecture}
\newtheorem{claim}[theorem]{Claim}

\newtheorem{definition}[theorem]{Definition}

\numberwithin{equation}{section}

\newcounter{capitalcounter}

\newcommand{\N}{\mathbb{N}}

\newcommand{\R}{\mathbb{R}}

\newcommand{\KK}{\mathcal{K}}

\newcommand{\HH}{\mathcal{H}}
\newcommand{\VV}{\mathcal{V}}
\renewcommand{\AA}{\mathcal A}
\newcommand{\BB}{\mathcal B}
\renewcommand{\d}{\delta}
\renewcommand{\a}{\alpha}

\renewcommand{\epsilon}{\varepsilon}
\newcommand{\e}{\epsilon}
\renewcommand{\subset}{\subseteq}

\newcommand{\COMMENT}[1]{}
\renewcommand{\COMMENT}{\footnote} 

\title{Fractional Clique Decompositions of Dense Partite Graphs}

\author{Richard Montgomery}
 \address{Trinity College, Cambridge, CB2 1TQ}
 \email{rhm34@cam.ac.uk}
\thanks{The research leading to these results was partially supported by the  European Research Council under the European Union's Seventh Framework Programme (FP/2007--2013) / ERC Grant Agreement n.\ 258345.}

\begin{document}
\date{\today}

\begin{abstract}
We give a minimum degree condition sufficent to ensure the existence of a fractional $K_r$-decomposition in a balanced $r$-partite graph (subject to some further simple necessary conditions). This generalises the non-partite problem studied recently by Barber, Lo, K\"uhn, Osthus and the author, and the $3$-partite fractional $K_3$-decomposition problem studied recently by Dukes. Combining our result with recent work by Barber, K\"uhn, Lo, Osthus and Taylor, this gives a minimum degree condition sufficient to ensure the existence of a (non-fractional) $K_r$-decomposition in a balanced $r$-partite graph (subject to the same simple necessary conditions).
\end{abstract}

\maketitle 

\section{Introduction}

Given a graph $F$, we say a graph $G$ has an \emph{$F$-decomposition} if there is a collection of edge-disjoint copies of $F$ in $G$ that covers all the edges of $G$. The study of $F$-decompositions dates back to 1847, when Kirkman~\cite{Kirkman} showed that the $n$-vertex clique~$K_n$ has a $K_3$-decomposition if and only if $n\equiv 1,3 \mod 6$. Much later, Wilson~\cite{Wilson} was able to determine whether a complete graph $K_n$ has an $F$-decomposition for any graph $F$, when $K_n$ is large compared to $F$. A hypergraph generalisation regarding the decomposition of large cliques into smaller cliques has only recently been achieved, in a breakthrough by Keevash~\cite{Keevash}.

Progress has also recently been made by Barber, K\"uhn, Lo and Osthus~\cite{BKLO} in finding $F$-decompositions of large graphs which are not complete, but have a high minimum degree. A key component in these new methods, as explained later, is to first find a relevant \emph{fractional decomposition} of the large graph.
A graph $G$ has a \emph{fractional $F$-decomposition} if a weighting can be given to the copies of $F$ in $G$ so that each edge lies in copies of $F$ with total weight $1$. That is, if $\mathcal{F}(G)$ is the set of copies of $F$ in $G$, then there is a function $\omega:\mathcal{F}(G)\to [0,1]$ so that, for each edge $e\in E(G)$, $\sum_{F\in \mathcal{F}(G):e\in E(F)}\omega(F)=1$. 

For example, any clique $K_n$ can be seen to have a fractional $K_r$-decomposition if $n\geq r\geq 2$ by simply weighting all the copies of $K_r$ in $K_n$ by $1/\binom{n-2}{r-2}$. In fact, any large graph with a sufficiently high minimum degree has a fractional $K_r$-decomposition. First shown by Yuster~\cite{Yuster}, the required minimum degree was improved by Dukes~\cite{Dukes12,Dukescorr}, before Barber, K\"uhn, Lo, Osthus and the current author~\cite{BKLMO} showed that any graph $G$ on $n\geq 10^4r^3$ vertices with minimum degree $\delta(G)\geq (1-1/10^4r^{3/2})n$ has a fractional $K_r$-decomposition. On the other hand, Yuster~\cite{Yuster} has constructed graphs showing that for each $\e>0$ and integers $r$ and~$n_0$ there is, for some $n\geq n_0$, an $n$-vertex graph $G$ with minimum degree $\delta(G)\geq (r/(r+1)+\e)n$ without a fractional $K_r$-decomposition. In the particular case $r=3$, the minimum degree required to ensure a fractional $K_3$-decomposition was improved by Yuster~\cite{Yuster}, Dukes~\cite{Dukes12,Dukescorr} and Garaschuk~\cite{Garaschuk}, before Dross~\cite{Dross} proved that any $n$-vertex graph $G$ with minimum degree $\delta(G)\geq 9n/10$ has a fractional $K_3$-decomposition.

In this paper, we study the minimum degree required to ensure a fractional $K_r$-decomposition of \emph{$r$-partite} graphs, where, unlike in the non-partite case, the graph must necessarily satisfy a further simple condition. In combination with recent work by Barber, K\"uhn, Lo, Osthus and Taylor~\cite{BKLOT} this gives good bounds on the minimum degree required to ensure a $K_r$-decomposition of $r$-partite graphs satisfying the same necessary condition.

We say an $r$-partite graph $G$ on a vertex partition $(V_1,\ldots,V_r)$ is \emph{balanced} if $|V_1|=\ldots=|V_r|$. If such a graph $G$ has a fractional $K_r$-decomposition $\omega:\mathcal{F}(G)\to [0,1]$, where $\mathcal{F}(G)$ is the set of copies of $K_r$ in $G$, then for each $i\in[r]$, $v\in V_i$ and $j\in [r]\setminus \{i\}$ we have
\[
\sum_{K\in \mathcal{F}(G):v\in V(K)}\omega(K)=\sum_{u\in V_j\cap N(v)}\sum_{K\in \mathcal{F}(G):uv\in E(K)}\omega(K)=\sum_{u\in V_j\cap N(v)}1=d(v,V_j),
\]
where $d(v,V_j)$ is the number of neighbours of $v$ in $V_j$. Therefore, if $G$ has a fractional $K_r$-decomposition then the degree of each vertex $v\in V_i$ must be the same into each other vertex class.
Let us say that an $r$-partite graph~$G$ with partition $(V_1,\ldots,V_r)$ is \emph{$K_r$-divisible} if it has this property, that is, if for every $i,j\in [r]$ and $v\in V(G)\setminus (V_i\cup V_j)$ we have $d(v,V_{i})=d(v,V_{j})$. Thus, an $r$-partite graph must be $K_r$-divisible if it has a fractional $K_r$-decomposition. 

Given an $r$-partite graph $G$ with partition $(V_1,\ldots,V_r)$, let
\[
\hat{\d}(G)=\min\{d(v,V_j):j\in [r],v\in V(G)\setminus V_j\}.
\]
Note that if such a graph $G$ is $K_r$-divisible then the minimum degree of~$G$ is $(r-1)\hat{\d}(G)$.
We will show that if $G$ is a balanced $K_r$-divisible $r$-partite graph in which $\hat{\d}(G)$ is sufficiently high (though potentially distinctly smaller than the size of the vertex classes), then $G$ has a fractional $K_r$-decomposition.

\begin{theorem}\label{fracdecomp} Let $r\geq 3$ and $n\in \N$. If~$G$ is a $K_r$-divisible $r$-partite graph on $(V_1,\ldots,V_r)$, where $|V_1|=\ldots=|V_r|=n$ and $\hat{\d}(G)\geq (1-1/10^6r^{3})n$, then~$G$ has a fractional $K_r$-decomposition.
\end{theorem}

Barber, K\"uhn, Lo and Osthus~\cite{BKLO} have created a method of iterative absorption capable of turning approximate $F$-decompositions (edge-disjoint copies of $F$ in $G$ which cover most of the edges of $G$) of certain dense graphs into $F$-decompositions. Haxell and R\"{o}dl~\cite{HaRo} have shown that (roughly speaking) large dense graphs with fractional $K_{\chi(F)}$-decompositions have approximate $F$-decompositions, where $\chi(F)$ is the chromatic number of~$F$. Thus, a dense graph with a fractional $K_{\chi(F)}$-decomposition must have an approximate $F$-decomposition, and hence, subject to a simple necessary condition, an $F$-decomposition~\cite{BKLO}. Recently, Barber, K\"uhn, Lo, Osthus and Taylor~\cite{BKLOT} have adapted and extended the methods in~\cite{BKLO} in order to apply them to $r$-partite graphs. In combination with Theorem~\ref{fracdecomp}, this gives the following result.

\begin{theorem}\label{decomp} For every $r\geq 3$ and $\e>0$, there exists an $n_0\in \N$ such that the following holds for all $n\geq n_0$. If~$G$ is a $K_r$-divisible $r$-partite graph on $(V_1,\ldots,V_r)$, where $|V_1|=\ldots =|V_r|=n$ and $\hat{\d}(G)\geq (1-1/10^6r^{3}+\e)n$, then~$G$ has a $K_r$-decomposition.
\end{theorem}

Dukes~\cite{Dukespart} has shown that if a $K_3$-divisible $3$-partite graph $G$ with $n$ vertices in each class satisfies $\hat{\delta}(G)\geq 101n/104$, then $G$ has a fractional $K_3$-decomposition. This is a better bound on the required minimum degree than that given in Theorem~\ref{fracdecomp} when $r=3$, and thus, when combined with the work of Barber, K\"uhn, Lo, Osthus and Taylor~\cite{BKLOT}, results in a better bound than that given in Theorem~\ref{decomp} when $r=3$. As noted in~\cite{Dukespart}, the methods introduced by Dukes can be used more generally to find fractional $K_r$-decompositions of dense $r$-partite graphs when $r\geq 3$. While this gives a better bound than that given in Theorem~\ref{fracdecomp} for small values of $r$, for large values of~$r$ this will give a weaker bound~\cite{Dukespart}.
In Theorem~\ref{decomp}, the case where $r=3$ is particularly interesting. As a corollary we may deduce that partially completed latin squares in which each symbol, row and column is used a limited number of times can be completed~\cite{BKLOT,Dukespart}. When $r\geq 4$, Theorem~\ref{decomp} permits a similar deduction to be made about the completion of $r-2$ mutually orthogonal latin squares (see~\cite{BKLOT} for details).

The $r$-partite version of the fractional $K_r$-decomposition problem can be viewed as a generalisation of the comparable non-partite problem, in the following way. Given a graph~$G$ with minimum degree $\delta(G)$, take $r$ disjoint copies of $V(G)$ to get the vertex set of a new graph $\hat{G}$, and let two vertices from different copies of $V(G)$ be connected by an edge in $\hat{G}$ if there is an edge between the corresponding vertices in~$G$. Considering this construction, we can see that $\hat{G}$ is $K_r$-divisible, $\hat{\delta}(\hat{G})= \delta(G)$, and~$\hat{G}$ has a fractional $K_r$-decomposition if and only if $G$ does. Thus, if the graph~$G$ has a sufficiently high minimum degree then we can apply Theorem~\ref{fracdecomp} to $\hat{G}$ to show that $G$ has a fractional $K_r$-decomposition.

To prove Theorem~\ref{fracdecomp}, we will use an idea introduced by Dross~\cite{Dross} and developed in a more general setting by Barber, K\"uhn, Lo, Osthus, and the author~\cite{BKLMO}.
Roughly speaking, to find a fractional $K_r$-decomposition of a (non-partite or $r$-partite) graph~$G$ we begin by uniformly weighting the copies of $K_r$ in $G$, so that the weight on the individual edges (defined as the sum of the weights of the copies of $K_r$ containing that edge) is on average 1. The weight on some edges will be greater than 1, and the weight on some edges will be less than 1, but if we have a sufficiently strong minimum degree condition then the weight on each edge will be close to 1 (as each edge is in a similar number of copies of $K_r$).
Furthermore, each copy of $K_r$ in $G$ has a strictly positive weight, allowing us to both increase and decrease the weights of the copies of $K_r$ while maintaining a non-negative weighting. In making such changes we aim to correct the weight on each edge to get a fractional $K_r$-decomposition. In~\cite{BKLMO}, the corrections were made to the weight on each edge $e\in E(G)$ in turn, making sure that at each stage only the weight on $e$ was adjusted, not the weight on any other edges. To make the corrections in~\cite{BKLMO}, it was critical that in a non-partite graph with a high minimum degree each edge was in many copies of $K_{r+2}$. As there are no copies of $K_{r+2}$ in an $r$-partite graph, we will need a new method to make these adjustments. In fact, we are unable to adjust the weight on an individual edge without changing the weight on some other edges and we will therefore make adjustments to the weight on multiple edges simultaneously. A sketch of our method is given in Section~\ref{sec:sketch}.


For each $s\geq r$, we could also ask more generally what minimum degree is needed in a balanced $s$-partite graph to ensure a fractional $K_r$-decomposition (subject perhaps to some necessary conditions). If $s\geq r+2$ then a comparable minimum degree condition to that in Theorem~\ref{fracdecomp} will ensure each edge is in many copies of $K_{r+2}$, whereupon the methods in~\cite{BKLMO} can be used directly. In particular, this method could find a fractional $K_r$-decomposition using a minimum degree bound depending on $r$, but not on $s$. When $s=r+1$ the situation is more complicated, and there may not be a simple set of divisibility conditions distinguishing which large graphs with a high minimum degree have a fractional $K_r$-decomposition.

The authors of~\cite{BKLMO} developed the basic method outlined above to reduce the minimum degree required in the non-partite setting.
It is likely that improvements along these lines could be made to our methods here to improve Theorem~\ref{fracdecomp}. However, these improvements would neither introduce any new ideas nor achieve a plausibly optimal bound, while obscuring the necessary changes due to the partite setting. Therefore, we will limit ourselves to a brief discussion of these possibilities in Section~\ref{sec:improve}.

Barber, K\"uhn, Lo, Osthus and Taylor~\cite{BKLOT} have conjectured that the minimum degree bound in Theorem~\ref{decomp} could be replaced by $\hat{\delta}(G)\geq (1-1/(r+1))n$ when~$n$ is sufficiently large. 
To show this would be optimal, they exhibited a balanced $r$-partite graph $G$ with $rn$ vertices and $\hat{\delta}(G)= \lceil(1-1/(r+1))n\rceil-1$ which has no $K_r$-decomposition. The same graph also has no fractional $K_r$-decomposition (see~\cite[Section 3.1]{BKLOT}), and in light of this we make the following, weaker, conjecture.

\begin{conjecture}\label{conj}
For every $r\geq 3$ there exists an $n_0\in \N$ such that the following holds for all $n\geq n_0$. If $G$ is a $K_r$-divisible graph on $(V_1,\ldots, V_r)$, where $|V_1|=\ldots=|V_r|=n$ and $\hat{\delta}(G)\geq (1-1/(r+1))n$, then $G$ has a fractional $K_r$-decomposition.
\end{conjecture}

\noindent
The results of Barber, K\"uhn, Lo, Osthus and Taylor~\cite{BKLOT} are sufficiently strong that a proof of Conjecture~\ref{conj} would be enough to show that, for each $\e>0$, any sufficiently large balanced $K_r$-divisible $r$-partite graph $G$ with $rn$ vertices and $\hat{\delta}(G)\geq (1-1/(r+1)+\e)n$ has a $K_r$-decomposition.

After detailing some of the notation we will use, in Section~\ref{sec:sketch} we sketch the main details of our proof before giving an overview of the rest of the paper.

\subsection{Notation}
We work with an $r$-partite graph~$G$, on the partition $(V_1,\ldots,V_r)$, where the sets $V_1,\ldots,V_r$ form a partition of $V(G)$ and there are no edges between any two vertices from the same set $V_i$, $i\in[r]=\{1,\ldots,r\}$. We denote by~$K_r$ the complete graph, or \emph{clique}, with $r$ vertices. We refer to the copies of $K_r$ in the graph~$G$ as the \emph{$r$-cliques} in $G$.

For a vertex $x\in V(G)$, we let $N(x)=\{y\in V(G):xy\in E(G)\}$ and $N^c(x)=V(G)\setminus N(x)$.
For a vertex set $X\subset V(G)$ in a graph $G$, $d(v,X)$ is the number of neighbours of $v$ in $X$. In an $r$-partite graph $G$ on $(V_1,\ldots,V_r)$, we let $\hat{\delta}(G)=\min\{d(v,V_j):j\in [r], v\in V(G)\setminus V_j\}$. Note that the value of $\hat{\delta}(G)$ depends on the partition of $G$, and therefore when we use it without defining a partition $(V_1,\ldots, V_r)$ we do so implicitly.
For a vertex set $X\subset V(G)$, we denote the graph induced on~$G$ by~$X$ as~$G[X]$.
By a \emph{weighting of the $r$-cliques in $G$}, we mean a function $\omega:\KK_r\to[0,1]$, where $\KK_r$ is the set of $r$-cliques in $G$, and we say the resulting \emph{weight on an edge $e\in E(G)$} is $\sum_{K\in \KK_r:e\in E(K)}\omega(K)$.

Given functions $f,g:\N\to[0,\infty)$, we say $f(n)=O(g(n))$ if there exists a constant $C>0$ such that $f(n)\leq Cg(n)$ for all $n\in \N$, and we say $f(n)=\Theta(g(n))$ if there exist constants $c,C>0$ such that $cg(n)\leq f(n)\leq Cg(n)$ for all $n\in \N$.
Finally, given any event $A$, we let
\[
\mathbf{1}_A=\left\{\begin{array}{ll}1 &\text{ if $A$ occurs} \\ 0  &\text{ if $A$ does not occur.} \end{array}\right.
\]

\section{Proof Sketch}\label{sec:sketch}
At the highest level, our proof follows the method introduced by Dross~\cite{Dross} and developed by Barber, K\"uhn, Lo, Osthus and the author~\cite{BKLMO}. We seek a fractional $K_r$-decomposition of a balanced $r$-partite $K_r$-divisible graph $G$ in which $\hat{\delta}(G)$ is close to the number of vertices in each class. We begin by uniformly weighting each copy of $K_r$, or $r$-clique, in $G$, so that the weight on the individual edges (taken as the sum of the weights of the $r$-cliques containing that edge) is on average 1. Due to the minimum degree condition, each edge will be in roughly the same number of $r$-cliques, so the weight on each edge will be close to 1 (as proved in Lemma~\ref{cliqonedge}). We aim to correct the weight on each individual edge to 1, by making small adjustments to the weight of the $r$-cliques. Each $r$-clique initially has a strictly positive weight, so these adjustments may be positive or negative, as long as the total adjustment to the weight of any $r$-clique is not too large.

We will break down the corrections needed to the weight on the edges into a sequence of smaller corrections which alter the weight on small groups of edges. We will make these smaller corrections using functions we call~\emph{gadgets}.
We call any function $f:\mathcal{K}_r\to \mathbb{R}$ a~\emph{gadget}, where $\KK_r$ is the set of $r$-cliques in $G$. Adding a gadget~$f$ to a weighting of the $r$-cliques will adjust the weight on each edge $e\in E(G)$ by $\xi_e:=\sum_{K\in \KK_r:e\in E(K)}f(K)$. For each $i\in [r]$, $v\in V_i$ and $j\in [r]\setminus \{i\}$, we have
\[
\sum_{u\in N(v)\cap V_j}\xi_{uv}=\sum_{u\in N(v)\cap V_j}\sum_{K\in \KK_r:uv\in E(K)}f(K)=\sum_{K\in \KK_r:v\in V(K)}f(K),
\]
as each $r$-clique contains exactly one vertex from each class. Therefore, the sum $\sum_{u\in N(v)\cap V_j}\xi_{uv}$ does not depend on $j$. Considering this, we can see that we cannot have a gadget that changes only the weight on one edge without altering the weight on some other edges. In fact, we will use two different gadgets which alter the weight on different collections of edges.

Our first gadget works with distinct vertices $v$ and $v'$ in some class $V_i$ and vertices $u_j\in V_j\cap N(v)\cap N(v')$, $j\in [r]\setminus \{i\}$; the gadget adds weight $w$ to (the weight on) each edge $vu_j$ and removes weight $w$ from each edge $v'u_j$ (see Figure~\ref{gadget1}). Our second gadget works with distinct vertices $v$ and $v'$ in some class $V_i$ and distinct vertices $u_1$ and $u_2$ in some other class $V_j$ which are both neighbours of $v$ and $v'$; the gadget removes weight $w$ from  $vu_2$ and $v'u_1$ and adds weight $w$ to $vu_1$ and $v'u_2$ (see Figure~\ref{gadget2}). Using these two gadgets, for any vertex $v$ we will be able to correct the weight on the edges incident to $v$, but in doing so we will also alter the weight on some edges incident to some other vertices in the same class as $v$ (which take the role of $v'$ in the gadgets). We think of this as moving the corrections we need to make to other edges. There will not typically be one vertex $v'$ to whose incident edges we can move all the corrections from the edges incident to $v$, as not all the neighbours of $v$ will be neighbours of some vertex $v'$.

\begin{figure}[b!]
\centering

\vspace{1cm}

\begin{tikzpicture}[scale=1]

\foreach \x in {1,...,4}
{
\draw[densely dashed,gray] ({2*\x},1.4) -- ({2*\x},3.6);
\draw[densely dashed,gray] ({2*\x+1.5},1.4) -- ({2*\x+1.5},3.6);
\draw[densely dashed,gray] ({2*\x+1.5},3.6)
arc [start angle=-0, end angle=180,
x radius=0.75cm, y radius=0.75cm];
\draw[densely dashed,gray] ({2*\x},1.4)
arc [start angle=-180, end angle=0,
x radius=0.75cm, y radius=0.75cm];

\draw ({2*\x+0.75},{-0.1+0.15}) node {$V_\x$};
};

\draw [fill] (2.75,1.5) circle [radius=0.1];
\draw [fill] (2.75,3.5) circle [radius=0.1];
\draw [fill] (2.75,1.1) node {$v'$};
\draw [fill] (2.75,3.9) node {$v$};
\draw (3.75,2.8) node {$+$};
\draw (4.45,2.9) node {$+$};
\draw (4.9,3.3) node {$+$};
\draw (3.75,2.15) node {$-$};
\draw (4.45,2.05) node {$-$};
\draw (4.9,1.75) node {$-$};

\foreach \x in {2,...,4}
{
\draw [fill] ({2*\x+0.75},2.5) circle [radius=0.1];
\draw ({2*\x+1.1},2.45) node {$u_\x$};
\draw (2.75,1.5)--({2*\x+0.75},2.5)--(2.75,3.5);
};

\end{tikzpicture}

\caption{The first gadget increases the weight on $vu_2, \ldots, vu_r$ and decreases the weight on $v'u_2, \ldots, v'u_r$, as depicted with $r=4$.}
\label{gadget1}
\end{figure}
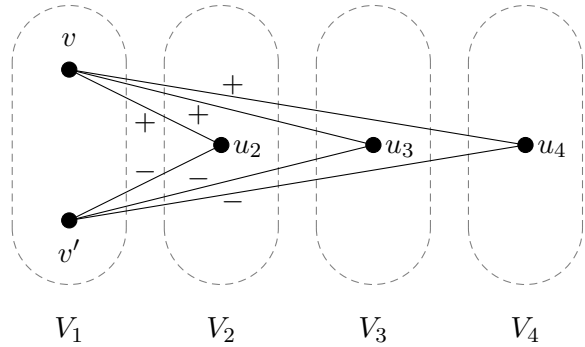

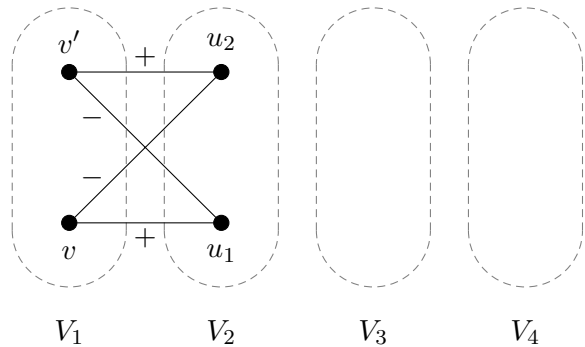
\begin{figure}[b!]
\centering

\vspace{1cm}

\begin{tikzpicture}[scale=1]

\foreach \x in {1,...,4}
{
\draw[densely dashed,gray] ({2*\x},1.4) -- ({2*\x},3.6);
\draw[densely dashed,gray] ({2*\x+1.5},1.4) -- ({2*\x+1.5},3.6);
\draw[densely dashed,gray] ({2*\x+1.5},3.6)
arc [start angle=-0, end angle=180,
x radius=0.75cm, y radius=0.75cm];
\draw[densely dashed,gray] ({2*\x},1.4)
arc [start angle=-180, end angle=0,
x radius=0.75cm, y radius=0.75cm];

\draw ({2*\x+0.75},{-0.1+0.15}) node {$V_\x$};
};

\draw [fill] (2.75,1.5) circle [radius=0.1];
\draw [fill] (2.75,3.5) circle [radius=0.1];
\draw [fill] (4.75,1.5) circle [radius=0.1];
\draw [fill] (4.75,3.5) circle [radius=0.1];
\draw [fill] (2.75,1.1) node {$v$};
\draw [fill] (2.75,3.9) node {$v'$};
\draw (4.75,3.5)--(2.75,3.5)--(4.75,1.5)--(2.75,1.5)--(4.75,3.5);
\draw [fill] (4.75,1.1) node {$u_1$};
\draw [fill] (4.75,3.9) node {$u_2$};
\draw (3.75,3.7) node {$+$};
\draw (3.05,2.9) node {$-$};
\draw (3.75,1.3) node {$+$};
\draw (3.05,2.1) node {$-$};

\end{tikzpicture}

\caption{The second gadget decreases the weight on $vu_2$ and $v'u_1$ and increases the weight on $vu_1$ and $v'u_2$, as depicted with $r=4$.}
\label{gadget2}
\end{figure}

We wish to move the required corrections to an area of the graph where they will be easier to make. Picking one specific $r$-clique $K$ from $G$, we aim to move the corrections onto the edges of $K$, where, as the sum of the corrections to be made is~$0$, they will naturally cancel out. Using our gadgets, we cannot typically move the corrections required to the edges around each vertex $v$ onto the edges incident to the vertex $v'$ in $K$ which is in the same class as $v$. Instead, where the class containing $v$ is $V_i$, we will move the corrections required to the edges incident to many different vertices in $V_i$, where the vertices in $V_i$ will be chosen to be neighbours of every vertex in $V(K)\setminus V_i$. Once we have done this for every vertex $v$ in the graph we will then use a second round of movements to move the corrections onto the edges of $K$, where they will naturally cancel out to give a fractional $K_r$-decomposition.
Finally, to ensure the above scheme does not alter the weight of one clique unduly much, we will take the average of the resulting gadget over all the $r$-cliques $K$ in the graph.

The rest of the paper is structured as follows. In Section~\ref{sec:cliqno}, we show that each edge in our graph is in roughly the same number of $r$-cliques. In Section~\ref{sec:gadgets}, we construct our two gadgets and break down the required corrections to the weight on the edges from the initial weighting into smaller corrections that can be made using the gadgets. In Section~\ref{sec:moveweight}, we use this to move the required corrections onto the edges of an $r$-clique~$K$, before averaging the resulting gadget over each different $r$-clique $K$ to prove Theorem~\ref{fracdecomp}. In Section~\ref{sec:improve}, we discuss possible improvements to Theorem~\ref{fracdecomp}.

When discussing these techniques, we will always have in mind a $K_r$-divisible $r$-partite graph~$G$ on $(V_1,\ldots,V_r)$ with $|V_1|=\ldots =|V_r|=n$, where the $r$-cliques have been given a uniform weight which we wish to correct to a fractional $K_r$-decomposition.




\section{Numbers of cliques}\label{sec:cliqno}
We will first prove some simple results concerning the number of cliques in our $r$-partite graph $G$. We wish to consider the number of cliques with vertices in specified classes, and therefore make the following definition.

\begin{definition}
Given an $r$-partite graph $G$ on $(V_1,\ldots,V_r)$ and a set $I\subset [r]$, let $\KK_I(G)$ be the set of $|I|$-cliques in $G$ with one vertex in $V_i$ for each $i\in I$. Where~$G$ is the only graph under consideration, let $\KK_I=\KK_{I}(G)$. Let $k_I=k_{I}(G)=|\KK_{I}(G)|$.
\end{definition}

In graphs with a high minimum degree, we can show that for each $I\subset [r]$ and $i\in I$, the quantities $k_I$ and $k_{I\setminus \{i\}}$ are closely related.

\begin{proposition}\label{cliqnos}
Let $r\ge 3$ and $n\in \N$, $0\leq \d\leq 1/2r$, and let~$G$ be an $r$-partite graph on $(V_1,\ldots,V_r)$, where $|V_1|=\ldots=|V_r|=n$ and $\hat{\d}(G)\geq (1-\d)n$. Then, for each $I\subset[r]$ and $i\in I$, 
\[
k_I/n\leq k_{I\setminus\{i\}}\leq (1+2\d r)k_I/n.
\] 
\end{proposition}
\begin{proof}
For each clique $K\in \KK_{I\setminus\{i\}}$, using the minimum degree of~$G$, the number of cliques in $\KK_I$ containing $K$ is at least $n-|I\setminus\{i\}|\cdot \d n\geq (1-\d r)n$, and at most $n$. Each clique $K\in \KK_I$ contains exactly one clique in $\KK_{I\setminus\{i\}}$. Therefore,
\[
k_I/n\leq k_{I\setminus\{i\}}\leq k_I/(1-\d r)n\leq (1+2\d r)k_I/n,
\]
where the last inequality follows as $\d r\leq 1/2$.
\end{proof}

We will now show that each edge in our graph is in approximately the same number of $r$-cliques. This will imply that the initial weighting of the $r$-cliques is close to a fractional $K_r$-decomposition.

\begin{lemma}\label{cliqonedge}
Let $r\ge 3$ and $n\in \N$, $0\leq \d\leq 1/8r$, and let~$G$ be an $r$-partite graph on $(V_1,\ldots,V_r)$, where $|V_1|=\ldots=|V_r|=n$ and $\hat{\d}(G)\geq (1-\d)n$. For each $e\in E(G)$, let $z_e$ be the number of $r$-cliques in $G$ containing $e$. Then, for each $e\in E(G)$, we have
\begin{equation}\label{cms1}
\Big|z_{e}-k_{[r]}/n^2\Big|\leq 9\d r k_{[r]}/n^2.
\end{equation}
\end{lemma}
\begin{proof} Let $i,j\in I$ and $e=xy\in E(G)$ with $x\in V_i$ and $y\in V_j$ and note that $z_{xy}$ is equal to the number of cliques in $\KK_{[r]\setminus\{i,j\}}$ which lie in $N(x)\cap N(y)$. If $k\in [r]\setminus\{i,j\}$ and $z\in V_k$, then the number of cliques in $\KK_{[r]\setminus\{i,j\}}$ which contain~$z$ is at most
\[
k_{[r]\setminus\{i,j,k\}}\leq (1+2\d r)^3k_{[r]}/n^3\leq (5/4)^3k_{[r]}/n^3\leq 2k_{[r]}/n^3,
\]
where we have used Proposition~\ref{cliqnos}. Therefore,
\begin{equation}\label{cms2}
|z_{xy}-k_{[r]\setminus\{i,j\}}|\leq \sum_{k\in[r]\setminus\{i,j\}}|(N^c(x)\cup N^c(y))\cap V_k|\cdot 2k_{[r]}/n^3\leq 4\d r k_{[r]}/n^2.
\end{equation}

By Lemma~\ref{cliqnos}, $|k_{[r]\setminus\{i,j\}}-k_{[r]}/n^2|\leq ((1+2\d r)^2-1)k_{[r]}/n^2\leq 5\d rk_{[r]}/n^2$, which, together with~\eqref{cms2}, implies~\eqref{cms1}.
\end{proof}


\section{Gadgets}\label{sec:gadgets}
In this section, we will construct certain functions, called \emph{gadgets}, which alter the weight on edges in our graph by altering the weight of the $r$-cliques. We will use gadgets which do not alter the total weight on the edges, for which we use the following definition.
\begin{definition}
Given any set $A$, a function $f:A\to \R$ is a \emph{zero-sum function} if $\sum_{a\in A}f(a)=0$.
\end{definition}

Our first gadget works with the edges between two distinct vertices~$v$ and $v'$ in some class $V_i$ and some vertices $u_j\in V_j\cap N(v)\cap N(v')$, $j\in [r]\setminus \{i\}$. Adding weight~$1$ to (the weight on) each edge $vu_j$, the gadget removes weight $1$ from each edge $v'u_j$ (as depicted in Figure~\ref{gadget1}). If $u_{j_1}u_{j_2}$ was an edge for each distinct $j_1$ and $j_2$ in $[r]\setminus \{i\}$, then we could easily create such a gadget. Indeed, adding weight~$1$ to the clique with vertex set $\{v,u_j:j\in[r]\setminus\{i\}\}$, and removing weight $1$ from the clique with vertex set $\{v',u_j:j\in[r]\setminus\{i\}\}$ effects this change, as the change in the weight on each edge $u_{j_1}u_{j_2}$ is cancelled out. Let us call such a gadget a \emph{simple gadget}.

Typically, we will not have all such edges $u_{j_1}u_{j_2}$ in our graph. Instead, we will find a set of new vertices $A=\{a_j:j\in [r]\setminus \{i\}\}$ where each vertex $a_j$ is a neighbour of every other vertex $u_{j'}$, $a_{j'}$, $v$ and $v'$ except for $u_j$. For each $j\in [r]\setminus \{i\}$, let $A_j=A\setminus \{a_j\}$. For each $j\in[r]\setminus\{i\}$, as $G[\{v,u_j\}\cup A_j]$ and $G[\{v',u_j\}\cup A_j]$ are cliques, we can use a simple gadget to increase the weight on each edge between $v$ and $A_j\cup \{u_j\}$ by~$1$ and decrease the weight on each edge between $v'$ and $A_j\cup\{u_j\}$ by~$1$. In total, this increases the weight on each edge between $v$ and $\{u_j:j\in [r]\setminus \{i\}\}$ by~$1$ and decreases the weight on each edge between $v'$ and $\{u_j:j\in [r]\setminus \{i\}\}$ by~1, as required, but it also increases the weight on each edge between $v$ and $A$ by $r-2$ and decreases the weight on each edge between $v'$ and $A$ by $r-2$. However, as $G[\{v\}\cup A]$ and $G[\{v'\}\cup A]$ are both cliques we can reverse this last change using a simple gadget.

This describes the underlying method of our gadget, but changes the weight of a few cliques by a large amount, while we wish only to make a small adjustment to the weight of any clique. To avoid this, we will take an average of the above construction for our first gadget over all possible such vertex sets $A$.

\begin{lemma}\label{weightK2r} Let $r\geq 3$ and $n\ge 8r^2$. Let~$G$ be an $r$-partite graph on $(V_1,\ldots,V_r)$, where $|V_1|=\ldots=|V_r|=n$ and $\hat{\d}(G)\geq (1-1/8r^2)n$. Let $j\in [r]$, let $v,v'\in V_j$ with $v\neq v'$ and, for each $i\in [r]\setminus \{j\}$, let $u_i\in V_i\cap N(v)\cap N(v')$.

Then, there is a zero-sum function $\psi:\KK_{[r]}\to \R$ so that the following hold.
\begin{enumerate}[label =(\roman{enumi})]
\item For each $e\in E(G)$,
\begin{equation*}\label{sunday}
\sum_{K\in \KK_{[r]}:e\in E(K)}\psi(K)= \left\{
\begin{array}{ll}
1 & \text{ if }e=vu_i\text{ for some }i\in[r]\setminus\{j\} 
\\
-1 &\text{ if }e=v'u_i\text{ for some }i\in[r]\setminus\{j\} 
\\
0 & \text{ otherwise.}
\end{array}
\right.
\end{equation*}

\item For each $K\in \KK_{[r]}$, letting $V=\{u_i:i\in [r]\setminus\{j\}\}$, we have
\begin{equation*}\label{sundayb}
|\psi(K)|\leq \left\{
\begin{array}{ll}
2n^2/k_{[r]} & \text{ if }|V(K)\cap V|=1\text{ and }|V(K)\cap\{v,v'\}|=1 
\\
2rn/k_{[r]} &\text{ if }|V(K)\cap V|=0\text{ and }|V(K)\cap\{v,v'\}|=1
\\
0 & \text{ otherwise.}
\end{array}
\right.
\end{equation*}
\end{enumerate}
\end{lemma}
\begin{proof} Suppose, without loss of generality, that $j=r$, and let $V=\{u_1,\ldots,u_{r-1}\}$.
Let $\HH$ be the set of sets $A=\{a_1,\ldots,a_{r-1}\}$ with $a_i\in V_i$, $a_i\in \cap_{i'\in [r-1]\setminus \{i\}}N(u_{i'})$, and $a_i\in N(v)\cap N(v')$, for each $i\in[r-1]$, and $G[A]\in \KK_{[r-1]}$. 

For each $K\in \KK_{[r-1]}$, $V(K)\notin \HH$ if and only if $V(K)$ intersects with $N^c(u_i)\setminus V_i$, for some $i\in[r-1]$, or $N^c(v)\setminus V_r$, or $N^c(v')\setminus V_r$. For each $j\in [r-1]$ and $z\in V_j$, there are at most $k_{[r-1]\setminus \{j\}}\leq 2k_{[r-1]}/n$ cliques $K\in \KK_{[r-1]}$ containing $z$, where we have used Proposition~\ref{cliqnos}. As $\hat{\d}(G)\geq (1-1/8r^2)n$, for each $j\in[r]$ and $z\in V_j$ we have $|N^c(z)\setminus V_j|\leq (r-1)n/8r^2$, and hence
\begin{align}
|\HH|&\geq k_{[r-1]}-|(\cup_{i\in[r-1]}N^c(u_i)\setminus V_i)\cup (N^c(v)\setminus V_r)\cup (N^c(v')\setminus V_r)|\cdot 2k_{[r-1]}/n
\nonumber
\\
&\geq k_{[r-1]}-(r+1)\cdot(r-1)n/8r^2\cdot 2k_{[r-1]}/n\geq k_{[r-1]}/2\geq k_{[r]}/2n,\label{frost}
\end{align}
where we have used Proposition~\ref{cliqnos}.

For each clique $K\in \KK_{[r]}$, let $\a_K$ be the number of sets $A\in \HH$ for which $K\subset G[A\cup V\cup \{v,v'\}]$.
For each clique $K\in \KK_{[r]}$, let
\begin{equation}\label{salmon}
\phi(K)=\left\{\begin{array}{ll}
1 & \text{ if }|V(K)\cap V|=1\text{ and }v\in V(K)\\
-1 & \text{ if }|V(K)\cap V|=1\text{ and }v'\in V(K)\\
-(r-2) & \text{ if }V(K)\cap V=\emptyset\text{ and }v\in V(K)\\
r-2 & \text{ if }V(K)\cap V=\emptyset\text{ and }v'\in V(K)\\
0 & \text{ otherwise,}
\end{array}
\right.
\end{equation}
and let $\psi(K)=\a_K\phi(K)/|\HH|$. Note that for each clique $K\in \KK_{[r]}$ we cannot have both $v\in V(K)$ and $v'\in V(K)$, and therefore $\phi$ is well-defined. We will show that~$\psi$ satisfies the requirements of the lemma.

Firstly, let $A\in \HH$. For each clique $K\in \KK_{[r]}$ with $K\subset G[A\cup V\cup \{v\}]$ and $v\in V(K)$, if we switch $v$ for $v'$ then we alter only the sign of $\phi(K)$. That is, $\phi(G[(V(K)\setminus\{v\})\cup \{v'\}])=-\phi(K)$. Thus, as $\phi(K)=0$ if $V(K)\cap \{v,v'\}=\emptyset$, if $e\in E(G)$ with $V(e)\cap \{v,v'\}=\emptyset$, then
\begin{align*}
\sum_{K\subset G[A\cup V\cup \{v,v'\}]:e\in E(K)}&\phi(K)
\\
&\hspace{-1cm}=\sum_{K\subset G[A\cup V\cup \{v\}]:e\in E(K)}\phi(K)+\sum_{K\subset G[A\cup V\cup \{v'\}]:e\in E(K)}\phi(K)=0.
\end{align*}
If $e=vu_i$ with $i\in[r-1]$, then the only $K\in \KK_{[r]}$ with $K\subset G[A\cup V\cup \{v,v'\}]$, $e\in E(K)$ and $\phi(K)\neq 0$ is $G[(A\setminus \{a_i\})\cup V(e)]$. Therefore,
\[
\sum_{K\subset G[A\cup V\cup \{v,v'\}]:e\in E(K)}\phi(K)=1.
\]
Similarly, if $e=v'u_i$ with $i\in[r-1]$, then $\sum_{K\subset G[A\cup V\cup \{v,v'\}]:e\in E(K)}\phi(K)=-1$.

If $e=va_i$ with $i\in[r-1]$, then the only $K\in \KK_{[r]}$ with $K\subset G[A\cup V\cup \{v,v'\}]$, $e\in E(K)$ and $\phi(K)\neq 0$ are $G[A\cup\{v\}]$ and the cliques $G[(A\setminus \{a_{i'}\})\cup \{u_{i'},v\}]$ with $i'\in [r-1]\setminus\{i\}$, so that
\[
\sum_{K\subset G[A\cup V\cup \{v,v'\}]:e\in E(K)}\phi(K)=- (r-2)+(r-2)\cdot 1=0.
\]
Similarly, if $e=v'a_i$ with $i\in[r-1]$ then $\sum_{K\subset G[A\cup V\cup \{v,v'\}]:e\in E(K)}\phi(K)=0$. If $e\in E(G)$ with $V(e)\not\subset A\cup V\cup \{v,v'\}$ then there are no cliques $K\in \KK_{[r]}$ with $K\subset G[A\cup V\cup\{v,v'\}]$ and $e\in E(K)$.
Therefore, if for each $e\in E(G)$ we set
\[
I(e)=\left\{
\begin{array}{ll}
1 & \text{ if }e=vu_i\text{ for some }i\in[r-1],
\\
-1 & \text{ if }e=v'u_i\text{ for some }i\in[r-1],
\\
0 & \text{ otherwise,}
\\
\end{array}
\right.
\]
then for each $A\in \HH$ and $e\in E(G)$ we have
\[
\sum_{K\subset G[A\cup V\cup \{v,v'\}]:e\in E(K)}\phi(K)=I(e).
\]

Thus, for each $e\in E(G)$,
\begin{align}
\sum_{K\in \KK_{[r]}:e\in E(K)}\psi(K)&=\frac{1}{|\HH|}\sum_{K\in \KK_{[r]}:e\in E(K)}\sum_{A\in\HH:K\subset G[A\cup V\cup \{v,v'\}]}\phi(K)\nonumber
\\
&=\frac{1}{|\HH|}\sum_{A\in \HH}\sum_{K\subset G[A\cup V\cup \{v,v'\}]:e\in E(K)}\phi(K)\nonumber
\\
&=\frac{1}{|\HH|}\sum_{A\in \HH}I(e)=I(e),\label{cms3}
\end{align}
and therefore (i) holds. Note furthermore that~\eqref{cms3} implies that
\[
\binom{r}{2}\sum_{K\in \KK_{[r]}}\psi(K)=\sum_{e\in E(G)}\sum_{K\in \KK_{[r]}:e\in E(K)}\psi(K)=\sum_{e\in E(G)}I(e)=0,
\]
and thus $\psi$ is a zero-sum function.

Secondly, note that for each $K\in \KK_{[r]}$, we have by~\eqref{salmon} that
\begin{equation}\label{sunday2}
|\phi(K)|\leq \left\{
\begin{array}{ll}
1 & \text{ if }|V(K)\cap V|=1\text{ and }|V(K)\cap\{v,v'\}|=1
\\
r &\text{ if }|V(K)\cap V|=0\text{ and }|V(K)\cap\{v,v'\}|=1
\\
0 & \text{ otherwise.}
\end{array}
\right.
\end{equation}
If $K\in \KK_{[r]}$, $|V(K)\cap V|=1$ and $|V(K)\cap\{v,v'\}|=1$, then for each set $A\in \HH$ with $V(K)\subset A\cup V\cup\{v,v'\}$ we have $V(K)\setminus(V\cup \{v,v'\})\subset A$, $|V(K)\setminus(V\cup \{v,v'\})|=r-2$ and $G[A]\in \KK_{[r-1]}$; thus $\a_K\leq n$. If $K\in \KK_{[r]}$, $|V(K)\cap V|=0$ and $|V(K)\cap\{v,v'\}|=1$, then there is at most one set $A\in \HH$ with $V(K)\subset A\cup V\cup\{v,v'\}$, namely $A=V(K)\setminus \{v,v'\}$ if $A\subset\cap_{i\in[r-1]}(N(u_i)\cup V_i)$ and $A\subset N(v)\cap N(v')$; thus $\a_K\leq 1$. Together with~\eqref{sunday2},~\eqref{frost} and the definition of $\psi$, this gives (ii).
\end{proof}

Our second gadget works with distinct vertices~$v$ and~$v'$ in some class $V_i$ and distinct vertices $u_1$ and $u_2$ in some other class $V_j$ which are neighbours of both $v$ and~$v'$. The gadget removes weight $1$ from $vu_2$ and $v'u_1$ and adds weight $1$ to $vu_1$ and $v'u_2$ (see Figure~\ref{gadget2}). Simpler than the construction for the first gadget, for the construction of the second gadget we first find a vertex set $A=\{a_{i'}\in V_i:i'\in[r]\setminus\{i,j\}\}$ of neighbours of $v$, $v'$, $u_1$ and $u_2$, so that $G[A]$ is a clique. Adding weight~$1$ to the cliques $G[\{v,u_1\}\cup A]$ and $G[\{v',u_2\}\cup A]$ and removing weight~$1$ from the cliques $G[\{v,u_2\}\cup A]$ and $G[\{v',u_1\}\cup A]$ produces the required change. Similarly as in our construction of the first gadget, we wish to avoid making large adjustments to the weight of any clique. Therefore, we will take our second gadget to be the average of this construction over all the different possible such vertex sets~$A$.

\begin{lemma}\label{weight4cycle} Let $r\geq 3$ and $n\ge 16r$. Let~$G$ be an $r$-partite graph on $(V_1,\ldots,V_r)$, where $|V_1|=\ldots=|V_r|=n$ and $\hat{\d}(G)\geq (1-1/16r)n$. Let $i,j\in [r]$ and let $v,v'\in V_i$ and $u_1,u_2\in V_j$ be distinct vertices with $vu_1,vu_2,v'u_1,v'u_2\in E(G)$.

Then, there is a zero-sum function $\psi:\KK_{[r]}\to \R$ so that the following hold.
\begin{enumerate}[label =(\roman{enumi})]
\item For each $e\in E(G)$,
\begin{equation*}
\sum_{K\in \KK_{[r]}:e\in E(K)}\psi(K)= \left\{
\begin{array}{ll}
1 & \text{ if }e=vu_1\text{ or }v'u_2
\\
-1 &\text{ if }e=vu_2\text{ or }v'u_1
\\
0 & \text{ otherwise.}
\end{array}
\right.
\end{equation*}

\item For each $K\in \KK_{[r]}$, letting $V=\{v,v',u_1,u_2\}$, we have
\begin{equation*}\label{sundayc}
|\psi(K)|\leq \left\{
\begin{array}{ll}
2n^2/k_{[r]} & \text{ if }|V(K)\cap V|=2
\\
0 & \text{ otherwise.}
\end{array}
\right.
\end{equation*}
\end{enumerate}
\end{lemma}

\begin{proof}
Suppose, without loss of generality, that $i=r-1$ and $j=r$, and let $V=\{v,v',u_1,u_2\}$.
Let $\HH$ be the set of sets $A=\{a_1,\ldots,a_{r-2}\}$ with $a_i\in V_i$ and $a_i\in N(v)\cap N(v')\cap N(u_1)\cap N(u_2)$, for each $i\in[r-2]$, and $G[A]\in \KK_{[r-2]}$.

For each $K\in \KK_{[r-2]}$, $V(K)\notin \HH$ if and only if $V(K)$ intersects with $N^c(v)\cup N^c(v')\cup N^c(u_1)\cup N^c(u_2)$. For each $j\in [r-2]$ and $z\in V_j$, there are at most $k_{[r-2]\setminus \{j\}}\leq 2k_{[r-2]}/n$ cliques $K\in \KK_{[r-2]}$ containing $z$, where we have used Proposition~\ref{cliqnos}. Therefore, as $\hat{\d}(G)\geq (1-1/16r)n$,
\begin{align}
|\HH|&\geq k_{[r-2]}-|(N^c(v)\cup N^c(v')\cup N^c(u_1)\cup N^c(u_2))\setminus (V_{j-1}\cup V_j)|\cdot 2k_{[r-2]}/n
\nonumber
\\
&\geq k_{[r-2]}-4r\cdot n/16r\cdot 2k_{[r-2]}/n= k_{[r-2]}/2\geq k_{[r]}/2n^2,\label{frost2}
\end{align}
where we have again used Proposition~\ref{cliqnos}.

For each $K\in \KK_{[r]}$, let $\a_K$ be the number of sets $A\in \HH$ for which $K\subset G[A\cup V]$, let
\begin{equation}\label{rawfish}
\phi(K)=\left\{\begin{array}{ll}
1 & \text{ if }vu_1\text{ or }v'u_2\in E(K)\\
-1 & \text{ if }vu_2\text{ or }v'u_1\in E(K)\\
0 & \text{ otherwise,}
\end{array}
\right.
\end{equation}
and let $\psi(K)=\a_K\phi(K)/|\HH|$. We will show that $\psi$ satisfies the requirements of the lemma.

Firstly, let $A\in \HH$ and $e\in E(G)$ with $V(e)\subset A\cup V$. There are 4 cliques $K\in \KK_{[r]}$ for which $\phi(K)\neq 0$ and $K\subset G[A\cup V]$, namely $G[A\cup \{v,u_1\}]$, $G[A\cup \{v,u_2\}]$, $G[A\cup \{v',u_1\}]$ and $G[A\cup \{v',u_2\}]$. 

If $V(e)\subset A$, then $e$ is contained in each of these cliques, and hence we have that $\sum_{K\subset G[A\cup V]:e\in E(K)}\phi(K)=0$. If $|V(e)\cap V|=1$, then $e$ is contained in exactly two of these cliques -- the two cliques containing the single vertex in $V(e)\cap V$ -- and in each possibility we can see that $\sum_{K\subset G[A\cup V]:e\in E(K)}\phi(K)=0$.  If $V(e)\not\subset A\cup V$ then none of these cliques contain $e$.
If $V(e)\subset V$, then $e$ is contained in only one of these cliques and checking the possibilities we can see that 
\[
\sum_{K\subset G[A\cup V]:e\in E(K)}\phi(K)=\mathbf{1}_{\{e\in \{vu_1,v'u_2\}\}}-\mathbf{1}_{\{e\in \{vu_2,v'u_1\}\}},
\]
where from the reasoning above, this equation also holds for all other edges $e\in E(G)$.
Thus, for each $e\in E(G)$,
\begin{align}
\sum_{K\in \KK_{[r]}:e\in E(K)}\psi(K)&=\frac{1}{|\HH|}\sum_{K\in \KK_{[r]}:e\in E(K)}\sum_{A\in\HH:K\subset G[A\cup V]}\phi(K)\nonumber
\\
&=\frac{1}{|\HH|}\sum_{A\in \HH}\sum_{K\subset G[A\cup V]:e\in E(K)}\phi(K)\nonumber
\\
&=\frac{1}{|\HH|}\sum_{A\in \HH}\left(\mathbf{1}_{\{e\in \{vu_1,v'u_2\}\}}-\mathbf{1}_{\{e\in \{vu_2,v'u_1\}\}}\right)\nonumber
\\
&=\mathbf{1}_{\{e\in \{vu_1,v'u_2\}\}}-\mathbf{1}_{\{e\in \{vu_2,v'u_1\}\}},\label{final0}
\end{align}
and therefore (i) holds. Note furthermore that
\begin{align*}
\binom{r}{2}\sum_{K\in \KK_{[r]}}\psi(K)&=\sum_{e\in E(G)}\sum_{K\in \KK_{[r]}:e\in E(K)}\psi(K)\overset{\eqref{final0}}{=}0,
\end{align*}
and thus $\psi$ is a zero-sum function.

Secondly, for each $K\in \KK_{[r]}$,~\eqref{rawfish} implies that $|\phi(K)|= \mathbf{1}_{\{|V(K)\cap V|=2\}}$.
If $K\in \KK_{[r]}$ and $|V(K)\cap V|=2$, then there is at most one set $A\in \HH$ with $K\subset A\cup V$, namely $V(K)\setminus V$ if $V(K)\setminus V\subset \cap_{u\in V}N(u)$; thus $\a_K\leq 1$. Therefore, as $|\HH|\geq k_{[r]}/2n^2$ by~\eqref{frost2}, if $|V(K)\cap V|=2$ then $\psi(K)\leq 2n^2/k_{[r]}$. As for each $K\in \KK_{[r]}$ with $|V(K)\cap V|\neq 2$ we have $\psi(K)=0$, this gives (ii).
\end{proof}

Finally in this section, given a set of required corrections to the weights of edges around a vertex~$v$ (where the total corrections to the weights on the edges from~$v$ into any class is the same, as will naturally occur in our setup) we will break down these corrections into a set $\AA$ of corrections we can make using the first type of gadget and a set $\BB$ of corrections we can make using the second type of gadget.

\begin{lemma}\label{weightdecomp}  Let $r\geq 3$ and $n\in\N$. Let~$G$ be an $r$-partite graph on $(V_1,\ldots,V_r)$, where $|V_1|=\ldots=|V_r|=n$. Let $j\in[r]$, $v\in V_j$ and $z\in \R$. Let $z_{vu}\in \R$ for each $u\in N(v)$, and suppose that for each $i\in [r]\setminus\{j\}$ we have
\[
\sum_{u\in V_i\cap N(v)}z_{vu}=z.
\]
Let $\VV$ be the set of sets $V\subset N(v)$ with $|V\cap V_i|=1$ for each $i\in[r]\setminus\{j\}$.
Then, there are sets $\AA\subset \{(A,a):A\in \VV,a\in \R\}$ and $\BB\subset \{(u_1,u_2,b):u_1,u_2\in V_i\cap N(v)\text{ for some }i\in[r]\setminus\{j\},b>0\}$ so that the following hold.
\begin{enumerate}[label=\rm{(\roman{enumi})}]
\item \label{mary1} For each $(A,a)\in \AA$ and $u\in A$, $\mathrm{sgn}(z_{vu})=\mathrm{sgn}(a)$.
\item \label{mary2} For each $(u_1,u_2,b)\in \BB$, $z_{vu_1}>0$ and $z_{vu_2}<0$.
\item \label{mary3} For each $u\in N(v)$,
\[
z_{vu}=\sum_{(A,a)\in \AA}a\cdot\mathbf{1}_{\{u\in A\}}+\sum_{(u_1,u_2,b)\in \BB}b\cdot (\mathbf{1}_{\{u=u_1\}}-\mathbf{1}_{\{u=u_2\}}).
\]
\end{enumerate} 
\end{lemma}
\begin{proof} Suppose without loss of generality that $j=r$. We will prove the lemma by induction on the number of weights $z_{vu}$, $u\in N(v)$, which are non-zero. If all the weights are non-zero then we simply take $\AA=\emptyset$ and $\BB=\emptyset$.

Let us suppose then that there is some $i\in [r-1]$ and $u_i\in V_i\cap N(v)$ with $z_{vu_i}\neq 0$, and furthermore select such an $i$ and $u_i$ so that $|z_{vu_i}|$ is minimised. Assume that $z_{vu_i}>0$, where the case where $z_{vu_i}<0$ follows similarly.

If there is some $u_i'\in V_i\cap N(v)$ with $z_{vu_i'}<0$, then let $b=z_{vu_i}$. For each $u\in N(v)$, let
\[
z'_{vu}=\begin{cases}
z_{vu}-b & \text{ if }u=u_i \\ 
z_{vu}+b & \text{ if }u=u'_i \\ 
z_{vu} & \text{ otherwise. }
\end{cases}
\]
As $z'_{vu_i}=0$ and $z_{vu_i},z_{vu_i'}\neq 0$, there is at least one more non-zero weight $z_{vu}$ than non-zero weight $z'_{vu}$, and hence by the induction hypothesis there exist sets $\AA'\subset \{(A,a):A\in \VV,a\in \R\}$ and $\BB'\subset \{(u_1',u_2',b):u_1',u_2'\in V_j\text{ for some }j\in[r-1],b\in \R\}$ for which \ref{mary1}--\ref{mary3} hold with the weights $z'_{vu}$.
Note that, due to the choice of~$i$ and~$u_i$, $\mathrm{sgn}(z_{vu})=\mathrm{sgn}(z'_{vu})$ for each $u\in N(v)$. Let $\AA=\AA'$ and $\BB=\BB'\cup \{(u_i,u_i',b)\}$; \ref{mary1}--\ref{mary3} hold for $\AA$ and $\BB$ with the weights $z_{vu}$.

Therefore, we may assume that there is no $u_i'\in V_i\cap N(v)$ with $z_{vu_i'}<0$, and thus $z=\sum_{u\in V_i\cap N(v)}z_{vu}\geq z_{vu_i}>0$. Therefore, for each $j\in [r-1]\setminus\{i\}$, as $\sum_{u\in V_j\cap N(v)}z_{vu}=z>0$, we can find a vertex $u_j\in V_j\cap N(v)$ with $z_{vu_j}>0$. Let $a=z_{vu_i}$. 
 For each $u\in N(v)$, let
\[
z'_{vu}=\begin{cases}
z_{vu}-a & \text{ if }u=u_j\text{ for some }j\in[r-1] \\ 
z_{vu} & \text{ otherwise. }
\end{cases}
\]
As $z'_{vu_i}=0$ and $z_{vu_j}>0$ for each $j\in[r]$, there is at least one more non-zero weight~$z_{vu}$ than non-zero weight $z'_{vu}$, and thus by the induction hypothesis there exist sets $\AA'\subset \{(A,a):A\in \VV,a\in \R\}$ and $\BB'\subset \{(u_1',u_2',b):u_1',u_2'\in V_j\text{ for some }j\in[r-1],b\in \R\}$ for which \ref{mary1}--\ref{mary3} hold with the weights $z'_{vu}$.
Note that, due to the choice of~$i$ and $u_i$, $\mathrm{sgn}(z_{vu})=\mathrm{sgn}(z'_{vu})$ for each $u\in N(v)$. Let $\AA=\AA'\cup \{(\{u_j:j\in[r-1]\},a)\}$ and $\BB=\BB'$; \ref{mary1}--\ref{mary3} hold for $\AA$ and $\BB$ with the weights $z_{vu}$. This completes the inductive step, and hence the proof of the lemma.
\end{proof}


\section{Moving weight onto a clique}\label{sec:moveweight}

Having constructed our gadgets in the previous section, we will now use them to move weight around the graph. Given corrections we wish to make to the weight on the edges around a vertex $v$ in the graph, we will use the gadgets to make these corrections, while necessarily making changes to the weight on other edges. By ensuring that the vertices which take the role of $v'$ in our gadgets lie in some fixed subset $V$, we will be able to make all the required corrections to the weight on edges in the graph, except for edges with at least one endvertex in $V$. We will first give a definition which will allow us to specify sets $V$ into which we can move weight in this manner, before using our gadgets to make corrections to the weight on edges next to a vertex $v$ at the expense of alterations made to the weight on edges adjacent to vertices in the same class as $v$ and in $V$, proving Lemma~\ref{weightinto1set}.

\begin{definition} Let $r\geq 3$, let~$G$ be an $r$-partite graph on $(V_1,\ldots,V_r)$, and let $j\in[r]$. We say a set $V\subset V_j$ is \emph{$j$-neighbour-rich} if for each subset $W\subset V(G)\setminus V_j$ with $|W|\leq r$ we have $|V\cap (\cap_{u\in W}N(u))|\geq |V|/2$.
\end{definition}


\begin{lemma}\label{weightinto1set}  Let $r\geq 3$ and $n\ge 8r^2$. Let~$G$ be an $r$-partite graph on $(V_1,\ldots,V_r)$, where $|V_1|=\ldots=|V_r|=n$ and $\hat{\d}(G)\geq (1-1/8r^2)n$. Let $j\in[r]$, $v\in V_j$ and $z\in \R$. Let $z_{vu}\in [-1,1]$ for each $u\in N(v)$, and suppose that for each $i\in [r]\setminus\{j\}$ we have
\begin{equation*}
\sum_{u\in V_i\cap N(v)}z_{vu}=z.
\end{equation*}
Let $V\subset V_j$ be a $j$-neighbour-rich vertex set with $v\in V_j\setminus V$. 

Then, there is a  zero-sum function $\psi:\KK_{[r]}\to \R$ such that the following hold.
\stepcounter{capitalcounter}
\begin{enumerate}[label =\bfseries \Alph{capitalcounter}\arabic*]
\item For each $u\in N(v)$, we have $\sum_{K\in \KK_{[r]}:vu\in E(K)}\psi(K)=z_{vu}$.\label{tue1}
\item If $e\in E(G)$ and $V(e)\cap (V\cup \{v\})=\emptyset$ or $V(e)\cap N(v)=\emptyset$, then \label{tue2}
\[
\sum_{K\in \KK_{[r]}:e\in E(K)}\psi(K)=0.
\]
\item If $uw\in E(G)$ with $u\in V$ and $w\in N(v)$, then \label{tue3}
\[
\Big|\sum_{K\in \KK_{[r]}:uw\in E(K)}\psi(K)\Big|\leq 2|z_{vw}|/|V|.
\]
\item For each $K\in \KK_{[r]}$, letting $C=n\sum_{u\in V(K)\cap N(v)}|z_{vu}|+2\sum_{u\in N(v)}|z_{vu}|$, we have\label{tue4}
\[
|\psi(K)|\leq 
\left\{
\begin{array}{ll}
\frac{2nC}{k_{[r]}} & \text{ if }v\in V(K) 
\\
\frac{4nC}{|V|k_{[r]}} &\text{ if }V(K)\cap V\neq \emptyset 
\\
0 & \text{ otherwise.}
\end{array}
\right.
\]
\end{enumerate}
\end{lemma}

\begin{proof} Suppose without loss of generality that $j=r$. Let $\VV$ be the set of sets $A\subset N(v)$ with $|A\cap V_i|=1$ for each $i\in[r-1]$.
By Lemma~\ref{weightdecomp}, we may take sets $\AA\subset \{(A,a):A\in \VV,a\in \R\}$ and $\BB\subset \{(u_1,u_2,b):u_1,u_2\in V_i\cap N(v)\text{ for some }i\in[r-1],b>0\}$ which satisfy the following.
\stepcounter{capitalcounter}
\begin{enumerate}[label =\bfseries \Alph{capitalcounter}\arabic*]
\item \label{mary12} For each $(A,a)\in \AA$ and $u\in A$, $\mathrm{sgn}(z_{vu})=\mathrm{sgn}(a)$.
\item \label{mary22} For each $(u_1,u_2,b)\in \BB$,  $z_{vu_1}>0$ and  $z_{vu_2}<0$.
\item \label{mary32} For each $u\in N(v)$,
\[
z_{vu}=\sum_{(A,a)\in \AA}a\cdot\mathbf{1}_{\{u\in A\}}+\sum_{(u_1,u_2,b)\in \BB}b\cdot (\mathbf{1}_{\{u=u_1\}}-\mathbf{1}_{\{u=u_2\}}).
\]
\end{enumerate} 
Note that~\ref{mary12}-\ref{mary32} imply that for each $u\in N(v)$
\begin{equation}\label{sat0}
|z_{vu}|=\sum_{(A,a)\in \AA}|a|\cdot\mathbf{1}_{\{u\in A\}}+\sum_{(u_1,u_2,b)\in \BB}|b|\cdot |\mathbf{1}_{\{u=u_1\}}-\mathbf{1}_{\{u=u_2\}}|.
\end{equation}

For each $(A,a)\in \AA$, as $V$ is $r$-neighbour-rich, we have $|V\cap (\bigcap_{u\in A}N(u))|\geq |V|/2$.
For each $(A,a)\in \AA$, and each vertex $v'\in V\cap (\bigcap_{u\in A}N(u))\subset V_r$, by Lemma~\ref{weightK2r} we may let $\psi_{A,v'}:\KK_{[r]}\to \R$ be a zero-sum function satisfying the following.
\stepcounter{capitalcounter}
\begin{enumerate}[label =\bfseries \Alph{capitalcounter}\arabic*]
\item For each $e\in E(G)$,\label{sun1}
\begin{equation*}
\sum_{K\in \KK_{[r]}:e\in E(K)}\psi_{A,v'}(K)= \mathbf{1}_{\{V(e)\cap A\neq\emptyset\}}\cdot (\mathbf{1}_{\{v\in V(e)\}}-\mathbf{1}_{\{v'\in V(e)\}}).
\end{equation*}
\item For each $K\in \KK_{[r]}$,\label{sun2}
\begin{equation*}
|\psi_{A,v'}(K)|\leq \left\{
\begin{array}{ll}
2n^2/k_{[r]} & \text{ if }|V(K)\cap A|=1\text{ and }|V(K)\cap\{v,v'\}|=1 
\\
2rn/k_{[r]} &\text{ if }|V(K)\cap A|=0\text{ and }|V(K)\cap\{v,v'\}|=1
\\
0 & \text{ otherwise.}
\end{array}
\right.
\end{equation*} 
\end{enumerate}

For each $(u_1,u_2,b)\in \BB$, as $V$ is $r$-neighbour-rich, we have $|V\cap N(u_1)\cap N(u_2)|\geq |V|/2$.
For each $(u_1,u_2,b)\in \BB$ and each vertex $v'\in V\cap N(u_1)\cap N(u_2)$, by Lemma~\ref{weight4cycle} we may let $\psi_{u_1,u_2,v'}$ be a zero-sum function satisfying the following.
\stepcounter{capitalcounter}
\begin{enumerate}[label =\bfseries \Alph{capitalcounter}\arabic*]
\item For each $e\in E(G)$,\label{sun3}
\begin{equation*}
\sum_{K\in \KK_{[r]}:e\in E(K)}\psi_{u_1,u_2,v'}(K)= \mathbf{1}_{\{e\in\{vu_1,v'u_2\}\}}-\mathbf{1}_{\{e\in\{vu_2,v'u_1\}\}}.
\end{equation*}

\item For each $K\in \KK_{[r]}$, we have\label{sun4}
\begin{equation*}\label{sundayd}
|\psi_{u_1,u_2,v'}(K)|\leq 2n^2/k_{[r]}\cdot \mathbf{1}_{\{|V(K)\cap \{v,v',u_1,u_2\}|=2\}}.
\end{equation*}
\end{enumerate}

For each $K\in \KK_{[r]}$, let
\begin{equation}\label{train1}
\psi_1(K)=\sum_{(A,a)\in \AA}\frac{a}{|V\cap(\bigcap_{u\in A}N(u))|}\sum_{v'\in V\cap (\bigcap_{u\in A}N(u))}\psi_{A,v'}(K),
\end{equation}
\begin{equation}\label{train2}
\psi_2(K)= \sum_{(u_1,u_2,b)\in \BB}\frac{b}{|V\cap N(u_1)\cap N(u_2)|}\sum_{v'\in V\cap N(u_1)\cap N(u_2)} \psi_{u_1,u_2,v'}(K),
\end{equation}
and $\psi(K)=\psi_1(K)+\psi_2(K)$.
We will show that $\psi$ satisfies our requirements, noting first that as $\psi$ is a weighted sum of zero-sum functions it is itself a zero-sum function.

To prove that~\ref{tue1} holds note that, for each $u\in N(v)$, by~\eqref{train1} and~\ref{sun1},
\begin{align}
\sum_{K\in \KK_{[r]}:vu\in E(K)}&\psi_1(K)\nonumber
\\
&=\sum_{(A,a)\in \AA}\frac{a}{|V\cap(\bigcap_{z\in A}N(z))|}\sum_{v'\in V\cap (\bigcap_{z\in A}N(z))}\sum_{K\in \KK_{[r]}:vu\in E(K)}\psi_{A,v'}(K)\nonumber
\\
&=\sum_{(A,a)\in \AA}\frac{a}{|V\cap(\bigcap_{z\in A}N(z))|}\sum_{v'\in V\cap (\bigcap_{z\in A}N(z))}\mathbf{1}_{\{u\in A\}}\nonumber
\\
&=\sum_{(A,a)\in \AA}a\cdot\mathbf{1}_{\{u\in A\}},\label{lastone0}
\end{align}
and, by~\eqref{train2} and~\ref{sun3},
\begin{align}
\sum_{K\in \KK_{[r]}:vu\in E(K)}&\psi_2(K) \nonumber
\\
&=\sum_{(u_1,u_2,b)\in \BB}\frac{b}{|V\cap N(u_1)\cap N(u_2)|}\sum_{v'\in V\cap N(u_1)\cap N(u_2)} \sum_{K\in \KK_{[r]}:vu\in E(K)} \psi_{u_1,u_2,v'}(K)\nonumber
\\
&=\sum_{(u_1,u_2,b)\in \BB}\frac{b}{|V\cap N(u_1)\cap N(u_2)|}\sum_{v'\in V\cap N(u_1)\cap N(u_2)}  (\mathbf{1}_{\{u=u_1\}}-\mathbf{1}_{\{u=u_2\}})\nonumber
\\
&= \sum_{(u_1,u_2,b)\in \BB}b \cdot (\mathbf{1}_{\{u=u_1\}}-\mathbf{1}_{\{u=u_2\}}).\label{lastone}
\end{align}
Therefore, by~\eqref{lastone0},~\eqref{lastone},~\ref{mary32} and the definition of $\psi=\psi_1+\psi_2$, for each $u\in N(v)$ we have $\sum_{K\in \KK_{[r]}:vu\in E(K)}\psi(K)=z_{vu}$, and thus~\ref{tue1} holds.

To prove that~\ref{tue2} holds, let $e\in E(G)$ with $V(e)\cap (V\cup \{v\})=\emptyset$ or $V(e)\cap N(v)=\emptyset$. For each $(A,a)\in \AA$ and $v'\in V\cap (\bigcap_{u\in A}N(u))$, as $A\subset N(v)$ we either have that $V(e)\cap\{v,v'\}=\emptyset$ or $V(e)\cap A=\emptyset$. Therefore, by~\ref{sun1},
$\sum_{K\in \KK_{[r]}:e\in E(K)}\psi_{A,v'}(K)=0$. 
Similarly, for each $(u_1,u_2,b)\in \BB$ and $v'\in V\cap N(u_1)\cap N(u_2)$ we have either $V(e)\cap\{v,v'\}=\emptyset$ or $V(e)\cap\{u_1,u_2\}=\emptyset$, and thus by~\ref{sun3} that
$\sum_{K\in \KK_{[r]}:e\in E(K)}\psi_{u_1,u_2,v'}(K)=0$. Therefore, by~\eqref{train1},~\eqref{train2} and as $\psi=\psi_1+\psi_2$, we have
$\sum_{K\in \KK_{[r]}:e\in E(K)}\psi(K)=0$, and thus~\ref{tue2} holds.

To prove that~\ref{tue3} holds, let $uw\in E(G)$ with $u\in V$ and $w\in N(v)$. For each $(A,a)\in \AA$ and $v'\in V\cap (\bigcap_{z\in A}N(z))$, as $u\neq v$ and $u\notin A$, we have by~\ref{sun1} that
\begin{equation}\label{train3}
\sum_{K\in \KK_{[r]}:uw\in E(K)}\psi_{A,v'}(K)=-\mathbf{1}_{\{w\in A\}}\cdot\mathbf{1}_{\{u=v'\}},
\end{equation}
and for each $(u_1,u_2,b)\in \BB$ and $v'\in V\cap N(u_1)\cap N(u_2)$, we have by~\ref{sun3} that
\begin{equation}\label{train4}
\sum_{K\in \KK_{[r]}:uw\in E(K)}\psi_{u_1,u_2,v'}(K)=\mathbf{1}_{\{u=v'\}}\cdot (\mathbf{1}_{\{w=u_2\}}-\mathbf{1}_{\{w=u_1\}}).
\end{equation}
Therefore, by~\eqref{train1} and \eqref{train3} we have
\begin{align}
\Big|\sum_{K\in \KK_{[r]}:uw\in E(K)}&\psi_1(K)\Big|\nonumber
\\
&=\Big|\sum_{(A,a)\in \AA}\frac{a}{|V\cap(\bigcap_{z\in A}N(z))|}
\sum_{v'\in V\cap (\bigcap_{z\in A}N(z))}\mathbf{1}_{\{w\in A\}}\cdot \mathbf{1}_{\{u=v'\}}\Big|\nonumber
\\
&\leq\Big|\sum_{(A,a)\in \AA}\frac{a\cdot \mathbf{1}_{\{w\in A\}}}{|V\cap(\bigcap_{z\in A}N(z))|}\Big|
\leq \frac{2}{|V|}\sum_{(A,a)\in \AA}|a|\cdot \mathbf{1}_{\{w\in A\}},\label{sat1}
\end{align}
and, by~\eqref{train2} and \eqref{train4} we have
\begin{align}
\Big|\sum_{K\in \KK_{[r]}:uw\in E(K)}&\psi_2(K)\Big|\nonumber
\\
&=\Big|\sum_{(u_1,u_2,b)\in \BB}\frac{b}{|V\cap N(u_1)\cap N(u_2)|}\sum_{v'\in V\cap N(u_1)\cap N(u_2)} \mathbf{1}_{\{u=v'\}}\cdot (\mathbf{1}_{\{w=u_2\}}-\mathbf{1}_{\{w=u_1\}})\Big|\nonumber
\\
&\leq\Big|\sum_{(u_1,u_2,b)\in \BB}\frac{b\cdot(\mathbf{1}_{\{w=u_2\}}-\mathbf{1}_{\{w=u_1\}})}{|V\cap N(u_1)\cap N(u_2)|}\Big|\nonumber
\\
&\leq\frac{2}{|V|}\sum_{(u_1,u_2,b)\in \BB}|b|\cdot|\mathbf{1}_{\{w=u_2\}}-\mathbf{1}_{\{w=u_1\}}|.\label{sat2}
\end{align}
Combining~\eqref{sat0},~\eqref{sat1} and~\eqref{sat2} we have
\[
\Big|\sum_{K\in \KK_{[r]}:uw\in E(K)}\psi(K)\Big|\leq 2|z_{vw}|/|V|,
\]
and thus~\ref{tue3} holds.

We will now prove~\ref{tue4}. First note that if $K\in \KK_{[r]}$, $v\notin V(K)$ and $V(K)\cap V=\emptyset$, then by~\ref{sun2},~\ref{sun4},~\eqref{train1}, and~\eqref{train2}, and as $\psi=\psi_1+\psi_2$, we have $\psi(K)=0$. This leaves us with two cases to consider with $K\in \KK_{[r]}$, that is, when $V(K)\cap V\neq \emptyset$ and when $v\in V(K)$. Note that, as $V$ and $\{v\}$ are disjoint subsets of $V_r$, these cases do not intersect, and in the former case $|V(K)\cap V|=1$.

Suppose then that $|V(K)\cap V|=1$. By~\eqref{train1} and~\ref{sun2} we have
\begin{align}
|\psi_1(K)|&\leq \sum_{(A,a)\in \AA}\frac{|a|}{|V\cap(\bigcap_{u\in A}N(u))|}\sum_{v'\in V\cap (\bigcap_{u\in A}N(u))}|\psi_{A,v'}(K)| \nonumber
\\
&\leq \sum_{(A,a)\in \AA}\frac{2|a|}{|V|}\sum_{v'\in V\cap (\bigcap_{u\in A}N(u))}\mathbf{1}_{\{v'\in V(K)\}}\cdot\left(\mathbf{1}_{\{|V(K)\cap A|=1\}}\frac{2n^2}{k_{[r]}}+\mathbf{1}_{\{|V(K)\cap A|=0\}}\frac{2rn}{k_{[r]}}\right)
\nonumber
\\
&\leq \sum_{(A,a)\in \AA}\frac{4n^2|a|}{|V|k_{[r]}}|V(K)\cap A|+\sum_{(A,a)\in \AA}\frac{4rn|a|}{|V|k_{[r]}}
\nonumber
\\
&\leq \frac{4n^2}{|V|k_{[r]}}\sum_{(A,a)\in \AA}|a|\cdot\Big(\sum_{u\in V(K)}\mathbf{1}_{\{u\in A\}}\Big)+\frac{8n}{|V|k_{[r]}}\sum_{(A,a)\in \AA}|a|\cdot\Big(\sum_{u\in V(G)}\mathbf{1}_{\{u\in A\}}\Big),\label{sat3}
\end{align}
where we have used that if $(A,a)\in \AA$ then $\sum_{u\in V(G)}\mathbf{1}_{\{u\in A\}}=|A|=r-1\geq r/2$.
Furthermore, by~\eqref{train2} and~\ref{sun4} we have
\begin{align}
|\psi_2(K)|&\leq \sum_{(u_1,u_2,b)\in \BB}\frac{|b|}{|V\cap N(u_1)\cap N(u_2)|}\sum_{v'\in V\cap N(u_1)\cap N(u_2)} |\psi_{u_1,u_2,v'}(K)|\nonumber
\\
&\leq \sum_{(u_1,u_2,b)\in \BB}\frac{2|b|}{|V|}\sum_{v'\in V\cap N(u_1)\cap N(u_2)} \frac{2n^2}{k_{[r]}}\cdot \mathbf{1}_{\{|\{u_1,u_2,v'\}\cap V(K)|=2\}}\nonumber
\\
&\leq \frac{4n^2}{|V|k_{[r]}}\sum_{(u_1,u_2,b)\in \BB}|b|\cdot \sum_{u\in V(K)}|\mathbf{1}_{\{u=u_1\}}-\mathbf{1}_{\{u=u_2\}}|.\label{sat4}
\end{align}
Let $C=n\sum_{u\in V(K)\cap N(v)}|z_{vu}|+2\sum_{u\in N(v)}|z_{vu}|$. Combining~\eqref{sat0},~\eqref{sat3}, and~\eqref{sat4}, while noting that, if $(A,a)\in \AA$, then $A\subset N(v)$ and, if $(u_1,u_2,b)\in \BB$, then $u_1,u_2\in N(v)$, we have that
\begin{align*}
|\psi(K)|&\leq \frac{4n^2}{|V|k_{[r]}}\sum_{u\in V(K)\cap N(v)}|z_{vu}|+\frac{8n}{|V|k_{[r]}}\sum_{u\in N(v)}|z_{vu}|= \frac{4nC}{|V|k_{[r]}},
\end{align*}
and thus~\ref{tue4} holds in the case where $|V(K)\cap V|=1$.

Now, for each clique $K\in \KK_{[r]}$ with $v\in V(K)$, by~\eqref{train1} and~\ref{sun2}, we have
\begin{align}
|\psi_1(K)|&\leq \sum_{(A,a)\in \AA}\frac{|a|}{|V\cap(\bigcap_{u\in A}N(u))|}\sum_{v'\in V\cap (\bigcap_{u\in A}N(u))}|\psi_{A,v'}(K)| \nonumber
\\
&\leq \sum_{(A,a)\in \AA}\frac{|a|}{|V\cap(\bigcap_{u\in A}N(u))|}\sum_{v'\in V\cap (\bigcap_{u\in A}N(u))}\left(\mathbf{1}_{\{|V(K)\cap A|=1\}}\frac{2n^2}{k_{[r]}}+\mathbf{1}_{\{|V(K)\cap A|=0\}}\frac{2rn}{k_{[r]}}\right)\nonumber
\\
&\leq \sum_{(A,a)\in \AA}\left(\frac{2n^2|a|}{k_{[r]}}|V(K)\cap A|+\frac{2rn|a|}{k_{[r]}}\right).\nonumber
\\
&\leq \frac{2n^2}{k_{[r]}}\sum_{(A,a)\in \AA}|a|\cdot\Big(\sum_{u\in V(K)}\mathbf{1}_{\{u\in A\}}\Big)
+\frac{4n}{k_{[r]}}\sum_{(A,a)\in \AA}|a|\cdot\Big(\sum_{u\in V(G)}\mathbf{1}_{\{u\in A\}}\Big),\label{sat5}
\end{align}
where we have again used that if $(A,a)\in \AA$ then $\sum_{u\in V(G)}\mathbf{1}_{\{u\in A\}}\geq r/2$.

Furthermore, by~\eqref{train2} and~\ref{sun4} we have
\begin{align}
|\psi_2(K)|&\leq \sum_{(u_1,u_2,b)\in \BB}\frac{|b|}{|V\cap N(u_1)\cap N(u_2)|}\sum_{v'\in V\cap N(u_1)\cap N(u_2)} |\psi_{u_1,u_2,v'}(K)|\nonumber
\\
&\leq \sum_{(u_1,u_2,b)\in \BB}\frac{|b|}{|V\cap N(u_1)\cap N(u_2)|}\sum_{v'\in V\cap N(u_1)\cap N(u_2)} \frac{2n^2}{k_{[r]}}\cdot \mathbf{1}_{\{|\{u_1,u_2\}\cap V(K)|=1\}}\nonumber
\\
&\leq \sum_{(u_1,u_2,b)\in \BB} \frac{2n^2|b|}{k_{[r]}}\cdot |\{u_1,u_2\}\cap V(K)|.\nonumber
\\
&\leq\frac{2n^2}{k_{[r]}}\sum_{(u_1,u_2,b)\in \BB}|b|\cdot\sum_{u\in V(K)}|\mathbf{1}_{\{u=u_1\}}-\mathbf{1}_{\{u=u_2\}}|\label{sat6}
\end{align}
Combining~\eqref{sat0},~\eqref{sat5}, and~\eqref{sat6}, while noting again that, if $(A,a)\in \AA$, then $A\subset N(v)$ and, if $(u_1,u_2,b)\in \BB$, then $u_1,u_2\in N(v)$, we have that
\begin{align*}
|\psi(K)|\leq \frac{2n^2}{k_{[r]}}\sum_{u\in V(K)\cap N(v)}|z_{vu}|+\frac{4n}{k_{[r]}}\sum_{u\in N(v)}|z_{vu}|= \frac{2nC}{k_{[r]}},
\end{align*}
and thus~\ref{tue4} also holds in the case where $v\in V(K)$, completing the proof of~\ref{tue4}.
\end{proof}

Lemma~\ref{weightinto1set} allows us to make corrections to the weights on edges incident to a vertex $v$, at the expense of changes to the weights on edges adjacent to vertices in a set $V$ which are in the same class as $V$. We will take a large set $V$ for which $V\cap V_j$ is $j$-neighbour-rich for each $j\in[r]$, and for each vertex outside $V$, lying say in $V_i$, we will use Lemma~\ref{weightinto1set} with $V\cap V_i$. This will allow us to make the correct adjustments to every edge which has no endvertex in $V$. Repeating a similar movement, we can then make the corrections to the weights on edges with exactly one endvertex in $V$. This will complete the required corrections to the weight on the edges not contained within the set $V$, giving us the following lemma.

\begin{lemma}\label{allweightinto1set} Let $r\geq 3$ and $n\ge 8r^2$. Let~$G$ be an $r$-partite graph on $(V_1,\ldots,V_r)$, where $|V_1|=\ldots=|V_r|=n$ and $\hat{\d}(G)\geq (1-1/8r^2)n$. Let $z_v\in \R$ for each $v\in V(G)$ and $z_{e}\in [-1,1]$ for each $e\in E(G)$. Suppose that for each $i\in[r]$ and $v\in V(G)\setminus V_i$ we have
\begin{equation}\label{lastminute2}
\sum_{u\in V_i\cap N(v)}z_{vu}=z_v.
\end{equation}
Let $V\subset V(G)$ be a set such that, for each $j\in [r]$, $V\cap V_j$ is $j$-neighbour-rich and $|V\cap V_j|=|V|/r$. 

Then, there is a zero-sum function $\psi:\KK_{[r]}\to \R$ such that the following hold.
\stepcounter{capitalcounter}
\begin{enumerate}[label =\bfseries \Alph{capitalcounter}\arabic*]
\item For each $e\in E(G)$ with $V(e)\not\subset V$, we have $\sum_{K\in \KK_{[r]}:e\in E(K)}\psi(K)=z_{e}$.\label{rand1}
\item For each $e\in E(G)$ with $V(e)\subset V$, we have 
\[
\Big|\sum_{K\in \KK_{[r]}:e\in E(K)}\psi(K)\Big|\leq 12n^2r^2/|V|^2.
\]\label{rand2}
\item For each $K\in \KK_{[r]}$, if $|V|\geq nr/2$, then $|\psi(K)|\leq 135n^2r^2/k_{[r]}$, and if $|V|\leq n$, then\label{rand3}
\begin{equation*}\label{sundaybbbbbbbb}
|\psi(K)|\leq \left\{
\begin{array}{ll}
15n^2r^2/k_{[r]} & \text{ if }V(K)\cap V=\emptyset
\\
51n^3r^2/k_{[r]}|V| &\text{ if }|V(K)\cap V|=1
\\
45n^4r^2|V(K)\cap V|^2/k_{[r]}|V|^2 & \text{ if }|V(K)\cap V|\geq 2.
\end{array}
\right.
\end{equation*}
\end{enumerate}
\end{lemma}

\begin{proof} For each $v\in V(G)$, let $j(v)\in[r]$ be such that $v\in V_{j(v)}$.
For each $w\in V(G)\setminus V$, by Lemma~\ref{weightinto1set} applied with the $j$-neighbour-rich set $V_w=V\cap V_{j(w)}$,  there exists a zero-sum function $\psi_w:\KK_{[r]}\to \R$, so that the following hold with 
\begin{equation}\label{latenight}
C_w=n\sum_{u\in V(K)\cap N(w)}|z_{wu}|+2\sum_{u\in N(w)}|z_{wu}|\leq 3nr.
\end{equation}
\stepcounter{capitalcounter}
\begin{enumerate}[label =\bfseries \Alph{capitalcounter}\arabic*]
\item For each $u\in N(w)$, we have $\sum_{K\in \KK_{[r]}:wu\in E(K)}\psi_w(K)=z_{wu}$.\label{tue12}
\item If $e\in E(G)$ and $V(e)\cap (V_w\cup \{w\})=\emptyset$ or $V(e)\cap N(w)=\emptyset$, then \label{tue22}
\[
\sum_{K\in \KK_{[r]}:e\in E(K)}\psi_w(K)=0.
\]
\item If $uv\in E(G)$ with $u\in V_w$ and $v\in N(w)$, then \label{tue32}
\[
\Big|\sum_{K\in \KK_{[r]}:uv\in E(K)}\psi_w(K)\Big|\leq 2|z_{wv}|/|V_w|\le 2r/|V|.
\]
\item For each $K\in \KK_{[r]}$, if $w\in V(K)$, then 
\[
|\psi_w(K)|\leq \frac{2nC_w}{k_{[r]}}\overset{\eqref{latenight}}{\leq} \frac{6n^2r}{k_{[r]}},
\] 
if $V(K)\cap V_w\neq\emptyset$, then 
\[
|\psi_w(K)|\leq \frac{4n C_w}{|V_w|k_{[r]}}\overset{\eqref{latenight}}{\leq}\frac{12n^2r^2}{|V|k_{[r]}},
\]
and if $V(K)\cap (\{w\}\cup V_w)=\emptyset$, then $\psi_w(K)=0$.\label{tue42}
\end{enumerate}
Now, for each $e\in E(G)$, let
\begin{equation}\label{stupid0}
z'_{e}=z_{e}-\frac{1}{2}\sum_{w\in V(G)\setminus V}\sum_{K\in\KK_{[r]}:e\in E(K)}\psi_w(K).
\end{equation}
\begin{claim}\label{claim0} For each $uv\in E(G)$, $|z'_{uv}|\leq 3nr/|V|$, and if $u,v\notin V$, then $z'_{uv}=0$.
\end{claim}
\begin{proof}[Proof of Claim~\ref{claim0}]
If $u,v\notin V$, then for each $w\in V(G)\setminus (V\cup \{u,v\})$ we have that $\{u,v\}\cap (V_{w}\cup\{w\})=\emptyset$ and so, by \ref{tue22}, $\sum_{K\in \KK_{[r]}:uv\in E(K)}\psi_{w}(K)=0$. By \ref{tue12}, we have $\sum_{K\in \KK_{[r]}:uv\in E(K)}\psi_v(K)=z_{uv}$ and $\sum_{K\in \KK_{[r]}:uv\in E(K)}\psi_u(K)=z_{uv}$. Therefore, by~\eqref{stupid0},
\begin{equation*}\label{stupid}
z'_{uv}=z_{uv}-\frac{1}{2}\sum_{K\in\KK_{[r]}:uv\in E(K)}(\psi_v(K)+\psi_u(K))=z_{uv}-\frac{1}{2}(z_{uv}+z_{uv})=0.
\end{equation*}

If $u\in V$ and $v\notin V$ then, by~\ref{tue22}, for each $w\in V(G)\setminus (V\cup\{v\})$ we have either $\sum_{K\in\KK_{[r]}:uv\in E(K)}\psi_{w}(K)=0$, or $u\in V_w\subset V_{j(w)}$ and $v\in N(w)$, in which case by~\ref{tue32} we have $|\sum_{K\in \KK_{[r]}:uv\in E(K)}\psi_{w}(K)|\leq 2r/|V|$. By~\ref{tue12}, $\sum_{K\in \KK_{[r]}:uv\in E(K)}\psi_v(K)=z_{uv}$ and therefore, by~\eqref{stupid0},
\begin{equation*}\label{stupid3}
|z'_{uv}|\leq |z_{uv}-z_{uv}/2|+n\cdot (2r/|V|)/2\leq 1+nr/|V|\leq 2nr/|V|.
\end{equation*}
Similarly, if $u\notin V$ and $v\in V$, then $|z'_{uv}|\leq 2nr/|V|$.

If $u,v\in V$, then by~\ref{tue22}, for each $w\in V(G)\setminus V$ with $\sum_{K\in \KK_{[r]}:uv\in E(K)}\psi_{w}(K)\neq 0$, we have either $u\in V_w$ and $v\in N(w)$ or $u\in V_w$ and $v\in N(w)$. In each case we must have $w\in V_{j(u)}\cup V_{j(v)}$ and, by~\ref{tue32}, $|\sum_{K\in \KK_{[r]}:uv\in E(K)}\psi_{w}(K)|\leq 2r/|V|$. Thus, by~\eqref{stupid0},
\begin{equation*}\label{stupid32}
|z'_{uv}|\leq |z_{uv}|+2n\cdot (2r/|V|)/2\leq 3nr/|V|.\qedhere
\end{equation*}
\end{proof}

For each $v\in V(G)$, let 
\begin{equation}\label{stupid2}
z'_v=z_v-\frac12\sum_{w\in V(G)\setminus V}\sum_{K\in \KK_{[r]}:v\in V(K)}\psi_{w}(K).
\end{equation}
For each $v\in V(G)$ and $i\in [r]\setminus \{j(v)\}$, we have from~\eqref{lastminute2},~\eqref{stupid0}, and~\eqref{stupid2} that
\begin{align*}
\sum_{u\in V_i\cap N(v)}z'_{vu}&=z_v-\frac12\sum_{u\in V_i\cap N(v)}\sum_{w\in V(G)\setminus V}\sum_{K\in\KK_{[r]}:uv\in E(K)}\psi_w(K)\\
&=z_v-\frac12\sum_{w\in V(G)\setminus V}\sum_{u\in V_i\cap N(v)}\sum_{K\in\KK_{[r]}:uv\in E(K)}\psi_w(K)\\
&=z_v-\frac12\sum_{w\in V(G)\setminus V}\sum_{K\in\KK_{[r]}:v\in V(K)}\psi_w(K)=z'_v.
\end{align*}
Thus, we can use Lemma~\ref{weightinto1set} with the weights $z'_{e}|V|/3nr$, $e\in E(G)$, and each set $V_w=V\cap V_{j(w)}$, $w\notin V$. That is, for each $w\notin V$ there exists a zero-sum function $\psi'_w:\KK_{[r]}\to \R$, so that the following hold, where for each $K\in \KK_{[r]}$ and $j\in[r]$ we let
\begin{equation}
C'_{K,j}=n|V(K)\cap(V\setminus V_{j})|+2|V|,\label{sunnight2}
\end{equation}
and for each $w\in V(G)\setminus V$ and $K\in \KK_{[r]}$ we let
\begin{align}
C'_{K,w}&=n\sum_{v\in V(K)\cap N(w)}|z'_{wv}|+2\sum_{v\in N(w)}|z'_{wv}|\nonumber
\\
&=n\sum_{v\in V(K)\cap N(w)\cap V}|z'_{wv}|+2\sum_{v\in N(w)\cap V}|z'_{wv}|
\overset{\eqref{sunnight2}}{\leq} \frac{3nr}{|V|}C'_{K,j(w)},\label{sunnight1}
\end{align}
where we have used Claim~\ref{claim0}.
\stepcounter{capitalcounter}
\begin{enumerate}[label =\bfseries \Alph{capitalcounter}\arabic*]
\item For each $u\in N(w)$, we have $\sum_{K\in \KK_{[r]}:wu\in E(K)}\psi'_w(K)=z'_{wu}$.\label{tue13}
\item If $e\in E(G)$ and $V(e)\cap (V_w\cup \{w\})=\emptyset$ or $V(e)\cap N(w)=\emptyset$, then \label{tue23}
\[
\sum_{K\in \KK_{[r]}:e\in E(K)}\psi'_w(K)=0.
\]
\item If $uv\in E(G)$ with $u\in V_w$ and $v\in N(w)$, then \label{tue33}
\[
\Big|\sum_{K\in \KK_{[r]}:uv\in E(K)}\psi'_w(K)\Big|\leq 2|z'_{wv}|/|V_w|=2r|z'_{wv}|/|V|.
\]
\item For each $K\in \KK_{[r]}$, if $w\in V(K)$ then \label{tue43}
\[
|\psi'_w(K)|\leq \frac{2nC'_{K,w}}{k_{[r]}}\overset{\eqref{sunnight1}}{\leq} \frac{6n^2rC'_{K,j(w)}}{|V|k_{[r]}},
\]
if $V(K)\cap V_w\neq\emptyset$, then 
\[
|\psi'_w(K)|\leq \frac{4nC'_{K,w}}{|V_w|k_{[r]}}\overset{\eqref{sunnight1}}{\leq} \frac{12n^2r^2C'_{K,j(w)}}{|V|^2k_{[r]}},
\]
and if $V(K)\cap(\{w\}\cup V_w)=\emptyset$ then $\psi'_w(K)=0$.
\end{enumerate}
For each $e\in E(G)$, let
\begin{equation}\label{mouse}
z''_{e}=z'_{e}-\sum_{w\in V(G)\setminus V}\sum_{K\in \KK_{[r]}:e\in E(K)}\psi'_{w}(K).
\end{equation}

\begin{claim} \label{claim1} For each $e\in E(G)$, if $V(e)\not\subset V$ then $z''_e=0$, and if $V(e)\subset V$ then $|z''_e|\leq 12n^2r^2|V|^2-9nr/|V|$.
\end{claim}
\begin{proof}[Proof of Claim~\ref{claim1}]
Let $uv\in E(G)$. Suppose that $u\notin V$ and $v\notin V$, so that, by Claim~\ref{claim0}, $z'_{uv}=0$. For each $w\in V(G)\setminus V$, we have by \ref{tue13}, \ref{tue23}, or \ref{tue33} that $\sum_{K\in \KK_{[r]}:uv\in E(K)}\psi'_w(K)=0$. Therefore, by~\eqref{mouse}, $z''_{uv}=z'_{uv}=0$.

Suppose that $u\notin V$ and $v\in V$. By~\ref{tue13}, $\sum_{K\in \KK_{[r]}:uv\in E(K)}\psi'_{u}(K)=z'_{uv}$. For each $w\in V(G)\setminus (V\cup V_{j(v)}\cup\{u\})$, as $\{u,v\}\cap (V_w\cup\{w\})=\emptyset$ we have by~\ref{tue23} that $\sum_{K\in \KK_{[r]}:uv\in E(K)}\psi'_{w}(K)=0$. For each $w\in V_{j(v)}\setminus (V\cup N(u))$, as $\{u,v\}\cap N(w)=\emptyset$ we have by~\ref{tue23} that $\sum_{K\in \KK_{[r]}:uv\in E(K)}\psi'_{w}(K)=0$. For each $w\in (V_{j(v)}\setminus V)\cap N(u)$, as $z'_{wu}=0$, we have by~\ref{tue33} that $\sum_{K\in \KK_{[r]}:uv\in E(K)}\psi'_{w}(K)=0$. Therefore, if $u\notin V$ and $v\in V$, then by~\eqref{mouse} $z''_{uv}=z'_{uv}-z'_{uv}=0$. Similarly, if $u\in V$ and $v\notin V$, then $z''_{uv}=0$.

Suppose that $u,v\in V$. By~\ref{tue23}, for each $w\in V(G)\setminus (V\cup V_{j(u)}\cup V_{j(v)})$ and each $w\in V(G)\setminus (V\cup N(u)\cup N(v))$ we have that $\sum_{K\in \KK_{[r]}:uv\in E(K)}\psi'_{w}(K)=0$. If $w\in V\cap V_{j(u)}$ and $w\in N(v)$, or if $w\in V\cap V_{j(v)}$ and $w\in N(u)$, then by~\ref{tue33} and Claim~\ref{claim0} we have 
\[
\Big|\sum_{K\in \KK_{[r]}:uv\in E(K)}\psi'_{w}(K)\Big|\leq 2r(3nr/|V|)/|V|=6nr^2/|V|^2.
\]
Therefore, by~\eqref{mouse},
\begin{equation*}\label{stupid4}
|z''_{uv}|\leq 3nr/|V|+2(n-|V|/r)\cdot 6nr^2/|V|^2=12n^2r^2/|V|^2-9nr/|V|.\qedhere
\end{equation*}
\end{proof}

For each $K\in \KK_{[r]}$, let $\psi(K)=\sum_{w\in V(G)\setminus V}(\psi_w(K)/2+\psi'_w(K))$, so that, by~\eqref{stupid0} and~\eqref{mouse}, for each $e\in E(G)$, $\sum_{K\in \KK_{[r]}:e\in E(K)}\psi(K)=z_e-z''_e$. If $V(e)\not\subset V$, then Claim~\ref{claim1} implies that~\ref{rand1} holds. If $V(e)\subset V$, then by Claim~\ref{claim1}, we have
$|\sum_{K\in \KK_{[r]}:e\in E(K)}\psi(K)|\leq |z_e|+|z''_e|\leq 12n^2r^2/|V|^2$, and thus~\ref{rand2} holds. Furthermore, as $\psi$ is a weighted sum of zero-sum functions it is itself a zero-sum function.

Finally, to prove~\ref{rand3}, let $K\in \KK_{[r]}$. By~\ref{tue42}, we have
\begin{align}
\sum_{w\in V(G)\setminus V}|\psi_w(K)| \nonumber
&\leq \sum_{w\in V(K)\setminus V}\frac{6n^2r}{k_{[r]}}
+
\sum_{w\in V(G)\setminus V:V(K)\cap V_w\neq \emptyset}\frac{12n^2r^2}{|V|k_{[r]}}
\\
&\leq r\cdot \frac{6n^2r}{k_{[r]}}
+
n\cdot |V(K)\cap V|\cdot \frac{12n^2r^2}{|V|k_{[r]}} \nonumber
\\
&= \frac{6n^2r^2}{k_{[r]}}\left(1+\frac{2n|V(K)\cap V|}{|V|}\right), \label{thurs1}
\end{align}
and, by~\ref{tue43} we have
\begin{align}
\sum_{w\in V(G)\setminus V}|\psi'_w(K)|
&\leq
\sum_{w\in V(K)\setminus V}\frac{6n^{2}r}{|V|k_{[r]}}C'_{K,j(w)}
+
\sum_{w\in V(G)\setminus V:V(K)\cap V_w\neq \emptyset}\frac{12n^{2}r^{2}}{|V|^{2}k_{[r]}}C'_{K,j(w)}.\label{tired}
\end{align}

Now, if $|V|\geq rn/2$, then by~\eqref{sunnight2} we have for each $j\in[r]$ that $C'_{K,j}\leq nr+ 2|V|\leq 4|V|$, and hence, by ~\eqref{tired},
\begin{align}
\sum_{w\in V(G)\setminus V}|\psi'_w(K)|
&\leq
r\cdot \frac{6n^{2}r}{|V|k_{[r]}}\cdot 4|V|+rn\cdot\frac{12n^{2}r^{2}}{|V|^{2}k_{[r]}}\cdot 4|V|\leq \frac{120n^2r^2}{k_{[r]}}. \label{tired3}
\end{align}
From~\eqref{thurs1} we have that
\begin{align*}
\frac{1}{2}\sum_{w\in V(G)\setminus V}|\psi_w(K)| 
&\leq  \frac{3n^2r^2}{k_{[r]}}\left(1+\frac{2n\cdot r}{rn/2}\right)=\frac{15n^2r^2}{k_{[r]}} .
\end{align*}
Together with the definition of $\psi$ and~\eqref{tired3}, we have that $|\psi(K)|\leq 135n^2r^2/k_{[r]}$, as required.

Suppose then that $|V|\leq n$. If $V(K)\cap V=\emptyset$, then, by~\eqref{sunnight2}, $C'_{K,j}=2|V|$ for each $j\in [r]$ and there are no vertices $w\in V(G)\setminus V$ with $V_w\cap V(K)\neq \emptyset$. Therefore, by~\eqref{thurs1},~\eqref{tired}, and the definition of $\psi$, we have
\begin{align}
|\psi(K)|
&\leq
\frac12\Big(\frac{6n^{2}r^2}{k_{[r]}}\Big)
+\sum_{w\in V(K)}\frac{6n^{2}r}{|V|k_{[r]}}\cdot 2|V|= \frac{15n^{2}r^2}{k_{[r]}}, \nonumber
\end{align}
as required.

Suppose that $|V(K)\cap V|=1$. For each $j\in[r]$ we have from~\eqref{sunnight2} that $C'_{K,j}\leq n+2|V|\leq 3n$. Furthermore, for each $w\in V(G)\setminus V$ with $V(K)\cap V_w\neq \emptyset$ we have $V(K)\cap V\cap V_{j(w)}\neq \emptyset$, and hence, as $|V(K)\cap V|=1$, $|V(K)\cap (V\setminus V_{j(w)})|=0$. Thus, by~\eqref{sunnight2},
$C'_{K,j(w)}=2|V|$.
Therefore, by~\eqref{tired} we have
\begin{align}
\sum_{w\in V(G)\setminus V}|\psi'_w(K)|
&\leq 
r\cdot \frac{6n^{2}r}{|V|k_{[r]}}\cdot 3n
+
n\cdot\frac{12n^{2}r^{2}}{|V|^{2}k_{[r]}}\cdot 2|V|= \frac{42n^3r^2}{|V|k_{[r]}}. \label{green1}
\end{align}
From~\eqref{thurs1} we have
\begin{align*}
\frac12\sum_{w\in V(G)\setminus V}|\psi_w(K)|
\leq \frac{3n^2r^2}{k_{[r]}}\left(1+\frac{2n}{|V|}\right)\leq \frac{9n^3r^2}{|V|k_{[r]}}. \nonumber
\end{align*}
Together with~\eqref{green1} and the definition of $\psi$, this gives $|\psi(K)|\leq 51n^3r^2/|V|k_{[r]}$, as required.

Suppose then that $|V(K)\cap V|\geq 2$. For each $j\in[r]$ we have from~\eqref{sunnight2} that $C'_{K,j}\leq n|V(K)\cap V|+2|V|\leq 2n|V(K)\cap V|$. Note that there are at most $n|V(K)\cap V|$ vertices $w\in V(G)\setminus V$ with $V(K)\cap V_w\neq\emptyset$.
Therefore, by~\eqref{tired}, and as $|V|\leq n$, we have
\begin{align}
\sum_{w\in V(G)\setminus V}|\psi'_w(K)|
&\leq
\left (r\cdot \frac{6n^{2}r}{|V|k_{[r]}}+n|V(K)\cap V|\cdot \frac{12n^{2}r^{2}}{|V|^{2}k_{[r]}}\right)\cdot 2n|V(K)\cap V| \nonumber
\\
&\leq \frac{36n^{4}r^2}{|V|^2k_{[r]}}|V(K)\cap V|^2.\label{green2}
\end{align}
Furthermore, as $|V|\leq n$, we have from~\eqref{thurs1} that
\begin{align*}
\frac12 \sum_{w\in V(G)\setminus V}|\psi_w(K)| \nonumber
\leq  \frac{3n^2r^2}{k_{[r]}}\left(\frac{3n|V(K)\cap V|}{|V|}\right)
\leq \frac{9n^4r^2}{|V|^2k_{[r]}}|V(K)\cap V|^2.
\end{align*}
Therefore, by~\eqref{green2} and the definition of $\psi$, $|\psi(K)|\leq 45n^4r^2|V(K)\cap V|^2/|V|^2k_{[r]}$, which completes the proof of~\ref{rand3}.
\end{proof}

Given a clique $K$ in the graph $G$, our aim is make all the corrections to the weight on the edges outside $K$, where if the sum of the corrections to made is $0$ then these corrections will naturally cancel out inside $K$ as well. 
We cannot use Lemma~\ref{allweightinto1set} to do this directly, as typically $V(K)\cap V_j$ will be not $j$-neighbour-rich for each $j\in [r]$. Indeed, this is equivalent to each vertex in $K$ being a neighbour of every vertex in the other classes.
Instead, we will use an intermediate set $V$ whose vertices are neighbours of each vertex in $V(K)$ except for the vertex in the same class. The set~$V$ will be large, so that the minimum degree condition for $G$ will imply that $V\cap V_j$ is $j$-neighbour-rich for each $j\in [r]$. We can then move the corrections needed onto the edges within $V$. Then, as $V(K)\cap V_j$ will be $j$-neighbour-rich in $G[V]$ for each $j\in [r]$, we can move the corrections into the set $V(K)$, to prove the following lemma.

\begin{lemma}\label{weightonto1clique} Let $r\geq 3$ and $n\ge 16r^2$. Let~$G$ be an $r$-partite graph on $(V_1,\ldots,V_r)$, where $|V_1|=\ldots=|V_r|=n$ and $\hat{\d}(G)\geq (1-1/16r^2)n$. For each $v\in V(G)$, let $z_v\in \R$, and for each $e\in E(G)$, let $z_{e}\in [-1,1]$. Suppose that for each $j\in[r]$, $i\in [r]\setminus\{j\}$ and $v\in V_j$ we have
\begin{equation}
\sum_{u\in N(v)\cap V_i}z_{uv}=z_v.\label{early}
\end{equation}
Suppose that, for each $i\in[r]$, $\sum_{v\in V_i}z_v=0$, and let $K\in \KK_{[r]}$. 

Then, there is a function $\phi:\KK_{[r]}\to \R$ such that the following hold.
\stepcounter{capitalcounter}
\begin{enumerate}[label = {\bfseries\Alph{capitalcounter}\arabic{enumi}}]
\item For each $e\in E(G)$, we have $\sum_{K'\in \KK_{[r]}:e\in E(K')}\phi(K')=z_{e}$.\label{fri1}
\item For each $K'\in \KK_{[r]}$,\label{fri2}
\begin{equation*}\label{sundaybbbb}
|\phi(K')|\leq \left\{
\begin{array}{ll}
10^3n^2r^2/k_{[r]} & \text{ if }V(K')\cap V(K)=\emptyset
\\
10^4n^3r/k_{[r]} &\text{ if }|V(K')\cap V(K)|=1
\\
10^4n^4|V(K')\cap V(K)|^2/4k_{[r]} & \text{ if }2\leq |V(K')\cap V(K)|\leq r.
\end{array}
\right.
\end{equation*}
\end{enumerate}
\end{lemma}

\begin{proof} Let $V(K)=\{v_1,\ldots,v_r\}$ with $v_i\in V_i$ for each $i\in[r]$. Pick a set $V\subset\cap_{i\in[r]}(V_i\cup N(v_i))$ satisfying $|V\cap V_i|=n(1-1/8r)$ for each $i\in[r]$ and $V(K)\subset V$, where this is possible as $\hat{\d}(G)\geq (1-1/16r^2)n$. For each $j\in [r]$ and $A\subset V(G)\setminus V_j$ with $|A|\leq r$ we have
\[
|V\cap V_j\cap (\cap_{a\in A}N(a))|\geq |V\cap V_j|-|A|\cdot n/16r^2\geq |V\cap V_j|/2.
\]
Therefore, for each $j\in [r]$, $V\cap V_j$ is $j$-neighbour-rich. Thus, by Lemma~\ref{allweightinto1set} there is a zero-sum function $\psi:\KK_{[r]}\to\R$ such that the following hold.
\stepcounter{capitalcounter}
\begin{enumerate}[label =\bfseries \Alph{capitalcounter}\arabic*]
\item For each $e\in E(G)$ with $V(e)\not\subset V$, we have $\sum_{K'\in \KK_{[r]}:e\in E(K')}\psi(K')=z_{e}$.\label{rand12}
\item For each $e\in E(G)$ with $V(e)\subset V$, we have 
\[
\Big|\sum_{K'\in \KK_{[r]}:e\in E(K')}\psi(K')\Big|\leq 12n^2r^2/|V|^2\leq 24.
\]\label{rand22}
\item For each $K'\in \KK_{[r]}$, we have\label{rand32} $|\psi(K')|\leq 135n^2r^2/k_{[r]}$.
\end{enumerate}
For each $e\in E(G)$, let $z'_e=z_e-\sum_{K'\in \KK_{[r]}:e\in E(K')}\psi(K')$, so that, by \ref{rand12}, if $V(e)\not\subset V$, then $z'_e=0$, and by~\ref{rand22}, if $V(e)\subset V$, then $|z'_e|\leq 1+24=25$.

For each $v\in V(G)$, let $z'_v=z_v-\sum_{K'\in \KK_{[r]}:v\in V(K')}\psi(K')$. Note that, as $\psi$ is a zero-sum function, we have for each $i\in[r]$ that
\begin{equation}\label{lastminute3}
\sum_{v\in V_i}z'_v=\sum_{v\in V_i}z_v-\sum_{v\in V_i}\sum_{K'\in\KK_{[r]}:v\in V(K')}\psi(K')=0-\sum_{K'\in \KK_{[r]}}\psi(K')=0.
\end{equation}

Take a new graph $G'=G[V]$, and note that if $K'\in \KK_{[r]}(G)\setminus \KK_{[r]}(G')$, then $K'$ contains a vertex in $V(G)\setminus V$. Therefore, if for each $v\in V(G)$ $j(v)$ is such that $v\in V_{j(v)}$, then
\[
|\KK_{[r]}(G)\setminus \KK_{[r]}(G')|\leq \sum_{v\in V(G)\setminus V}k_{[r]\setminus \{j(v)\}}(G)\leq (rn-|V|)\cdot 2k_{[r]}(G)/n\leq k_{[r]}(G)/2,
\] 
where we have used Proposition~\ref{cliqnos} and the fact that $|V|=rn(1-1/8r)$.
Thus, $k_{[r]}(G')\geq k_{[r]}(G)/2$.

For each $v\in V(G)$ and $i\in [r]\setminus\{j(v)\}$ we have, as $z'_{uv}=0$ if $u\notin V$,
\begin{align}
\sum_{u\in N(v)\cap V_i\cap V}z'_{uv}&=\sum_{u\in N(v)\cap V_i}z'_{uv}\nonumber
=\sum_{u\in N(v)\cap V_i}z_{uv}-\sum_{u\in N(v)\cap V_i}\sum_{K'\in \KK_r(G):uv\in V(K')}\psi(K')\\
&\overset{\eqref{early}}{=}z_v-\sum_{K'\in \KK_r(G):v\in V(K')}\psi(K')=z'_v,\label{green3}
\end{align}
so we may use Lemma~\ref{allweightinto1set} with the weights $z'_e/25$ in the graph~$G'$.

For each $j\in[r]$ and $v\in V\setminus V_j$, we have by the choice of $V$ that $v\in N(v_j)$. Therefore, as $V(G')=V$, $\{v_j\}$ is $j$-neighbour-rich in $G'$ for each $j\in[r]$. Thus, by Lemma~\ref{allweightinto1set} with the set $V(K)$ and the weights $z'_e/25$, there is a zero-sum function $\psi':\KK_{[r]}(G')\to\R$ such that the following hold
\stepcounter{capitalcounter}
\begin{enumerate}[label =\bfseries \Alph{capitalcounter}\arabic*]
\item If $e\in E(G)$ and $V(e)\not\subset V(K)$, then $\sum_{K'\in \KK_{[r]}(G'):e\in E(K')}\psi'(K')=z'_{e}$.\label{rand13}

\item For each $K'\in \KK_{[r]}$, we have, using $k_{[r]}=k_{[r]}(G)$, $|V(K)|=r$, and that $k_{[r]}(G')\geq k_{[r]}/2$,\label{rand33}
\begin{equation*}\label{sundaybbbbb}
|\psi'(K')|\leq \left\{
\begin{array}{ll}
25\cdot 30n^2r^2/k_{[r]} & \text{ if }V(K')\cap V(K)=\emptyset
\\
25\cdot 102n^3r/k_{[r]} &\text{ if }|V(K')\cap V(K)|=1
\\
25\cdot 90n^4|V(K')\cap V(K)|^2/k_{[r]} & \text{ if }|V(K')\cap V(K)|\geq 2.
\end{array}
\right.
\end{equation*}
\end{enumerate}
Extend the domain of $\psi'$ to $\KK_{[r]}=\KK_{[r]}(G)$ by setting $\psi'(K')=0$ for each $K'\in \KK_{[r]}\setminus \KK_{[r]}(G')$, and for each $e\in E(G)$ let $z''_e=z'_e-\sum_{K'\in\KK_{[r]}:e\in E(K')}\psi'(K')$. For each $e\notin E(K)$, by~\ref{rand13}, $z''_e=0$.

Then for each $i,j\in[r]$ with $i\neq j$ we have, as $\psi'$ is a zero-sum function,
\begin{align*}
z''_{v_iv_j}&=\sum_{u\in V_i}\sum_{v\in N(u)\cap V_j}z''_{uv}
=\sum_{u\in V_i}\sum_{v\in N(u)\cap V_j}z'_{uv}-\sum_{u\in V_i}\sum_{v\in N(u)\cap V_j}\sum_{K'\in\KK_{[r]}:uv\in E(K')}\psi'(K')\nonumber
\\
&\overset{\eqref{green3}}{=}\sum_{u\in V_i}z'_{u}-\sum_{K'\in\KK_{[r]}}\psi'(K')
 \overset{\eqref{lastminute3}}{=}0-0=0.
\end{align*}
Thus, if we set $\phi(K')=\psi(K')+\psi'(K')$ for each $K'\in \KK_{[r]}$, then~\ref{fri1} holds.


To show~$\psi$ satisfies our requirements, it is left to show that~\ref{fri2} holds. Let $K'\in \KK_{[r]}$ and recall that $n\geq 16r^2$. If $V(K')\cap V(K)=\emptyset$, then by~\ref{rand32} and~\ref{rand33} we have $|\phi(K)|\leq 10^3n^2r^2/k_{[r]}$. If $|V(K')\cap V(K)|=1$, then by~\ref{rand32} and~\ref{rand33} we have $|\phi(K)|\leq (2550+135r/n)n^3r/k_{[r]}\leq 10^4n^3r/k_{[r]}$. If $|V(K')\cap V(K)|\geq 2$, then by~\ref{rand32} and~\ref{rand33} we have $|\phi(K)|\leq (2250+135r^2/n^2)n^4|V(K')\cap V(K)|^2/k_{[r]}\leq 2500n^4|V(K')\cap V(K)|^2/k_{[r]}$.
\end{proof}

Finally, we can now prove Theorem~\ref{fracdecomp} by applying Lemma~\ref{weightonto1clique} with each of the different $r$-cliques~$K$ in the graph and taking an average of the resulting gadgets. This averaging ensures that the weight of any clique is not adjusted by a large amount.

\begin{proof}[Proof of Theorem~\ref{fracdecomp}]
Let $r\geq 3$ and $n\in \N$. Let~$G$ be an $r$-partite graph on $(V_1,\ldots,V_r)$, where $|V_1|=\ldots=|V_r|=n$ and $\hat{\d}(G)\geq (1-1/10^6r^3)n$. Note that if $n<10^6r^3$, then $\hat{\d}(G)=n$, so $G$ is a complete $r$-partite graph and hence trivially has a fractional $K_r$-decomposition. Suppose then that $n\geq 10^6r^3$.

For each $e\in E(G)$, let $z_e$ be the number of $r$-cliques in $G$ containing $e$, so that $z_e=|\{K\in \KK_{[r]}:e\in E(K)\}|$.
By Lemma~\ref{cliqonedge}, $|z_en^2/k_{[r]}-1|\leq 9/10^6r^2$ for each $e\in E(G)$. Due to the minimum degree of $G$, $n^2\binom{r}{2}\geq e(G)\geq n^2\binom r2(1-1/10^6r^3)$, and thus $|z_ee(G)/\binom r 2k_{[r]}-1|\leq 1/10^5r^2$.

For each $v\in V(G)$, let $z_v$ be the number of cliques containing $v$. Then, for each $i\in[r]$ and $v\in V(G)\setminus V_i$,
\begin{equation}\label{lastlastminute}
\sum_{u\in N(v)\cap V_i}\Big(z_{uv}e(G)/\binom r 2k_{[r]}-1\Big)=z_v e(G)/\binom r 2k_{[r]}-d(v,V_i),
\end{equation}
and thus for each distinct $i,j\in[r]$
\begin{align}
\sum_{v\in V_j}\sum_{u\in N(v)\cap V_i}&\Big(z_{uv}e(G)/\binom r 2k_{[r]}-1\Big)=\sum_{v\in V_j}\Big(z_ve(G)/\binom r 2k_{[r]}-d(v,V_i)\Big) \nonumber\\
&=k_{[r]} e(G)/\binom{r}{2}k_{[r]}-\sum_{v\in V_j}d(v,V_i)=e(G)/\binom{r}{2}-d(V_j,V_i),\label{nullthing}
\end{align}
where $d(V_j,V_i)$ is the number of edges between $V_j$ and $V_i$.
Now, for each $k\in[r]\setminus\{i,j\}$, as $G$ is $K_r$-divisible we have
\[
d(V_i,V_k)=\sum_{v\in V_i}d(v,V_k)=\sum_{v\in V_i}d(v,V_j)=d(V_i,V_j).
\]
Therefore there is the same number of edges between any two classes, that is, for each distinct $i,j\in[r]$ we have $e(G)=\binom{r}{2}d(V_i,V_j)$, and thus, by~\eqref{nullthing}, we have $\sum_{v\in V_j}\sum_{u\in N(v)\cap V_i}(z_{uv}e(G)/\binom r 2k_{[r]}-1)=0$.

Therefore, if we initially weight each clique by $e(G)/\binom{r}{2}k_{[r]}$, then, using~\eqref{lastlastminute}, the corrections $z_{e}e(G)/\binom{r}{2}k_{[r]}-1$, $e\in E(G)$, required to achieve a fractional $K_r$-decomposition satisfy the requirements for us to apply Lemma~\ref{weightonto1clique} for each $K\in \KK_{[r]}$.
That is, for each $K\in \KK_{[r]}$, there is some function $\psi_K:\KK_{[r]}\to\R$ such that for each $e\in E(G)$, $\sum_{K':e\in E(K')}\psi_K(K')=z_ee(G)/\binom r 2k_{[r]}-1$, and, for each $K'\in \KK_{r}$,~\ref{fri2} in Lemma~\ref{weightonto1clique} holds with the function $10^5r^2\psi_K$, so that
\begin{equation}\label{sundaybbb}
|\psi_K(K')|\leq \left\{
\begin{array}{ll}
n^2/10^2k_{[r]} & \text{ if }V(K')\cap V(K)=\emptyset
\\
n^3/10rk_{[r]} &\text{ if }|V(K')\cap V(K)|=1
\\
n^4/10r^2k_{[r]} & \text{ if }|V(K')\cap V(K)|=2
\\
n^4/40k_{[r]} & \text{ if }3\leq |V(K')\cap V(K)|\leq r.
\end{array}
\right.
\end{equation}
For each $K'\in \KK_{[r]}$, there are at most $\sum_{i\in[r]}k_{[r]\setminus\{i\}}\leq 2rk_{[r]}/n$ cliques $K\in \KK_{[r]}$ with $|V(K')\cap V(K)|=1$, where we have used Proposition~\ref{cliqnos}. Furthermore, there are at most $\sum_{i,j\in[r]:i\neq j}k_{[r]\setminus\{i,j\}}\leq \binom{r}{2}2^2k_{[r]}/n^2\leq 2r^2k_{[r]}/n^2$ cliques $K\in \KK_{[r]}$ with $|V(K')\cap V(K)|=2$, where we have used Proposition~\ref{cliqnos}. Similarly, there are at most $\binom{r}{3}2^3k_{[r]}/n^3\leq 2r^3k_{[r]}/n^3$ cliques $K\in \KK_{[r]}$ with $|V(K')\cap V(K)|\geq 3$.

Thus, combining these calculations with~\eqref{sundaybbb}, we have for each $K'\in \KK_{[r]}$ that
\begin{align*}
\sum_{K\in \KK_{[r]}}|\psi_K(K')|&\leq n^2/10^2k_{[r]}\cdot k_{[r]}+n^3/10rk_{[r]}\cdot 2rk_{[r]}/n
\\
&\qquad + n^4/10r^2k_{[r]}\cdot 2r^2k_{[r]}/n^2 +n^4/40k_{[r]}\cdot 2r^3k_{[r]}/n^3
\\
&= n^2/10^2+n^2/5+ n^2/5+r^3n/20\leq 4n^2/5\leq e(G)/\binom{r}{2},
\end{align*}
where we have used that $n\geq 10^6r^3$.
Therefore, if for each $K'\in \KK_{[r]}$ we let $w_{K'}=(e(G)/\binom{r}{2}-\sum_{K\in\KK_{[r]}}\psi_{K}(K'))/k_{[r]}$, then $w_{K'}\geq 0$. Furthermore, for each $e\in E(G)$, we have 
\begin{align*}
\sum_{K'\in\KK_{[r]}:e\in E(K')}w_{K'}&=\frac{z_ee(G)}{\binom{r}{2}k_{[r]}}-\frac{1}{k_{[r]}}\sum_{K\in\KK_{[r]}}\sum_{K'\in\KK_{[r]}:e\in E(K')}\psi_{K}(K')
\\
&=\frac{z_ee(G)}{\binom r2 k_{[r]}}-\frac{1}{k_{[r]}}\sum_{K\in\KK_{[r]}}\Big(\frac{z_ee(G)}{\binom r 2k_{[r]}}-1\Big)=1.
\end{align*}
Thus, the weights $w_{K'}$, $K'\in \KK_{[r]}$, form a fractional $K_r$-decomposition of $G$.
\end{proof}

\section{Limitations of our method and possible improvements}\label{sec:improve}
The methods used here to prove Theorem~\ref{fracdecomp} are comparable to those used by Barber, K\"uhn, Lo, Osthus and the author~\cite{BKLMO} to find a fractional $K_r$-decomposition in non-partite graphs with a high minimum degree (where the implementation is much simpler). In~\cite{BKLMO} these techniques were then developed to reduce the required minimum degree. In this section, we will discuss the limitations of the techniques used in this paper, and the possibility of implementing some of the improvements from~\cite{BKLMO}.


Let us suppose we have a $K_r$-divisible $r$-partite graph $G$ with $n$ vertices in each class and $\hat{\delta}(G)=(1-\delta)n$, for some $\delta=\delta(r)$. By initially weighting each $r$-clique uniformly so that the weight on the individual edges is on average 1, we ensure that the weight on each edge is within the interval $(1-9\d r,1+9\d r)$ (as follows from Lemma~\ref{cliqonedge}). This means that by moving a proportion at most $9\d r$ of the weight on the edges around we can gain a fractional $K_r$-decomposition. However, our gadgets with which we move the weight around are  not very efficient. Most of the change caused by altering the weight of an $r$-clique is cancelled out by altering the weight of other $r$-cliques, leaving changes to the weight on only a small number of \emph{altered edges}. Examining the gadgets, we see that we only alter the weight of cliques containing at most one altered edge. This means that all but at most a proportion $1/\binom{r}{2}$ of the changes made by altering the individual weight of an $r$-clique are cancelled out by other changes. Due to this inefficiency, we can (only) use the gadgets to move a proportion $\Theta(1/r^2)$ of the weight on the edges. By taking $\delta=\e /r^3$, for some small constant $\e>0$, we can successfully move up to a proportion $9\d r=\Theta(1/r^2)$ of the weight on the edges around, and thus can correct the initial weighting to gain a fractional $K_r$-decomposition, proving Theorem~\ref{fracdecomp}.

In order to reduce the minimum degree required by these methods we could either find a way to reduce the corrections we need to make to the initial weighting or find a more efficient way to make those corrections. 
In the non-partite setting the comparable initial method would show that an $n$-vertex graph $G$ with $\delta(G)\geq (1-\e/r^3)n$ must have a fractional $K_r$-decomposition, for some small constant $\e=\e(r)>0$. In~\cite{BKLMO}, the authors reduced the corrections needed to the initial weighting
by iteratively removing copies of $K_r$ from $G$ until this could not be done without breaching the minimum degree condition. 
They also moved the weight more efficiently around the graph by simultaneously altering the weight on every edge incident to a fixed vertex. In combination, this reduced the minimum degree bound required by the methods to $\delta(G)\geq (1-\e/r^{3/2})n$.

It seems likely that in the partite setting we can similarly reduce the amount of weight we need to move around the graph after the initial weighting, and thus reduce the bound required in Theorem~\ref{fracdecomp} to $\hat{\delta}(G)\geq (1-\e/r^2)n$ (some further details are given below). In the partite setting, the technicalities involved in attempting to adapt the second round of improvements from~\cite{BKLMO} in order to move the weight around more efficiently are significant, but it is plausible further progress can be made in this manner. However, even if this is possible it seems very unlikely the minimum degree bound in Theorem~\ref{fracdecomp} could be reduced to give the correct dependence on $r$ using such techniques, let alone that Conjecture~\ref{conj} could be proved.

We will conclude by sketching how the amount of weight that is needed to be moved from the initial weighting could be reduced, where the (substantial) remaining details can be inferred from~\cite{BKLMO}. Starting with a $K_r$-divisible $r$-partite graph $G$ satisfying $\hat{\delta}(G)\ge (1-\delta)n$, with $\delta=\e/r^2$ for some small constant $\e>0$, we iteratively remove copies of $K_r$ from~$G$ until no further copies can be removed without violating the minimum degree condition. For each $i\in[r]$, let $X_i\subset V_i$ be the set of vertices~$v$ for which $d(v,V_j)\geq (1-\d)n+1$ for each $j\in[r]\setminus\{i\}$. We must have $\min_i|X_i|\leq \d r n$, for otherwise (as follows from the minimum degree condition) there would be a copy of $K_r$ with a vertex each in the sets $X_i$, $i\in[r]$, which we could remove without breaking the minimum degree condition, a contradiction. Letting $j\in[r]$ be such that $|X_j|\leq \delta rn$, for each $i\in[r]\setminus\{j\}$ we have
\begin{equation}\label{smalldeviation}
\sum_{v\in V_j}|d(v,V_i)-(1-\delta)n|\leq |X_j|\d n+n\leq 2\delta^2 rn^2.
\end{equation}
As the graph $G$ is $K_r$-divisible, there is the same number of edges between any two different vertex classes. Due to the minimum degree condition this must be at least $(1-\delta)n^2$ edges and from~\eqref{smalldeviation} it is at most $(1-\delta)n^2+2\delta^2rn^2$ edges.
The number of $r$-cliques containing an edge $xy\in E(G)$ with $x\in V_i$ and $y\in V_j$ is related to $|(N^c(x)\cup N^c(y))\setminus (V_i\cup V_j)|$ (this could be shown using methods in the proof of Lemma~\ref{cliqonedge} with adaptations similar to those found in~\cite{BKLMO}), and 
\begin{align*}
|(N^c(x)&\cup N^c(y))\setminus (V_i\cup V_j)|\\
&=2(r-2)n-(r-2)(d(x)+d(y))/(r-1)+|(N^c(x)\cap N^c(y))\setminus(V_i\cup V_j)|.
\end{align*}
As $G$ is close to a complete $r$-partite graph, the average of  $|(N^c(x)\cap N^c(y))\setminus(V_i\cup V_j)|$ over all edges $xy$ is small, and on average the degree of $x$ and $y$ does not deviate far from $(1-\delta)n$. This limits the average deviation of $|(N^c(x)\cup N^c(y))\setminus (V_i\cup V_j)|$ from its typical value, and could be used to show that on average the corrections to be made to the initial weight on each edge is $O(\d^2r^2)$, while the maximum change required to the weight on any edge remains $O(\d r)$ (as follows from Lemma~\ref{cliqonedge}). 

As $\delta=\e/r^2$, for sufficiently small $\e>0$ this average change is less than the change that can be made using the gadgets, which was $\Theta(1/r^2)$. However, potentially the required correction to the weight on some edges may be well above average (yet still, as noted, $O(\delta r)$). If an $r$-clique contains too many of these edges then this method may change the weight of the clique too much, risking the weight of that clique becoming negative. To avoid this, we could use gadgets in our initial constructions in Lemmas~\ref{weightK2r} and~\ref{weight4cycle} which only alter the weight of those cliques in which the weight on their edges do not on average need large adjustments (similarly to the methods used in~\cite{BKLMO}).


\section*{Acknowledgements}
The author is grateful to Peter Dukes, Deryk Osthus and Daniela K\"uhn for helpful comments and discussions.

\bibliographystyle{plain}
\bibliography{rhmreferences}

\end{document}